\documentclass[11pt,letter]{article}%
\usepackage{amsmath}
\usepackage{amsfonts}
\usepackage{amssymb}
\usepackage{graphicx}
\usepackage{wrapfig}%
\newtheorem{theorem}{Theorem}

\newtheorem{corollary}[theorem]{Corollary}

\newtheorem{definition}[theorem]{Definition}
\newtheorem{example}[theorem]{Example}

\newtheorem{proposition}[theorem]{Proposition}
\newtheorem{remark}[theorem]{Remark}

\newenvironment{proof}[1][Proof]{\noindent\textbf{#1.} }{\ \rule{0.5em}{0.5em}}
\begin{document}

\title{Wave polynomials, transmutations and Cauchy's problem for the Klein-Gordon equation}
\author{Kira V. Khmelnytskaya$^{1}$, Vladislav V. Kravchenko$^{2}$, Sergii M.
Torba$^{2}$
\and and S\'{e}bastien Tremblay$^{3}$\thanks{Research was supported by CONACYT via
the project 166141, Mexico. Research of Sergii Torba was partially supported
by DFFD, Ukraine (GP/F32/030) and by SNSF, Switzerland (JRP IZ73Z0 of SCOPES
2009--2012). The research of S\'ebastien Tremblay is partly supported by grant from NSERC of Canada.}\\$^{1}${\footnotesize Faculty of Engineering, Autonomous University of
Queretaro, }\\{\footnotesize Cerro de las Campanas s/n, col. Las Campanas \ Quer\'{e}taro,
Qro. C.P. 76010 M\'{e}xico}\\$^{2}${\footnotesize \ Department of Mathematics, CINVESTAV del IPN, Unidad
Quer\'{e}taro }\\{\footnotesize \ Libramiento Norponiente \# 2000 Fracc. Real de Juriquilla
\ Quer\'{e}taro, Qro., CP 76230, M\'{e}xico}\\{\footnotesize vkravchenko@math.cinvestav.edu.mx}\\$^{3}${\footnotesize D\'{e}partement de math\'{e}matiques et d'informatique,
Universit\'{e} du Qu\'{e}bec,}\\{\footnotesize Trois-Rivi\`{e}res, Qu\'{e}bec, G9A 5H7, Canada}}
\maketitle

\begin{abstract}
We prove a completeness result for a class of polynomial solutions of the wave
equation called wave polynomials and construct generalized wave polynomials,
solutions of the Klein-Gordon equation with a variable coefficient. Using the
transmutation (transformation) operators and their recently discovered mapping
properties we prove the completeness of the generalized wave polynomials and
use them for an explicit construction of the solution of the Cauchy problem
for the Klein-Gordon equation. Based on this result we develop a numerical
method for solving the Cauchy problem and test its performance.

\end{abstract}

\bigskip

\section{Introduction}

Let $\Omega\subset\mathbb{C}$ be a simply connected domain. Due to the Runge
approximation theorem any harmonic in $\Omega$ function can be approximated
uniformly on any compact subset inside $\Omega$ by harmonic polynomials. The
harmonic polynomials are linear combinations of the polynomials
$\operatorname{Re}(z-z_{0})^{n}$ and $\operatorname{Im}(z-z_{0})^{n}$,
$n=0,1,\ldots,$ where $z_{0}$ is an arbitrary point in $\Omega$ and $z$ is a
complex variable. This fact reflects the completeness of the system of
harmonic polynomials $\left\{  \operatorname{Re}(z-z_{0})^{n},\
\operatorname{Im}(z-z_{0})^{n}\right\}  _{n=0}^{\infty}$ in the space of all
harmonic functions in $\Omega$ in the sense of the normal convergence.

Instead of the Laplace equation let us consider the wave equation
\begin{equation}
w_{xx}-w_{tt}=0\label{wave intro}%
\end{equation}
and instead of the complex imaginary unit let us introduce the hyperbolic
imaginary unit: $j^{2}=1$. Let $z$ denote the hyperbolic variable $z=x+jt$
\cite{Lavrentiev and Shabat}, \cite{Motter and Rosa 1998}. Analogously to the
elliptic case the system of polynomials
\begin{equation}
\big\{\operatorname{Re}(x+jt)^{n}\quad\text{and}\quad\operatorname{Im}%
(x+jt)^{n}\big\}_{n=0}^{\infty}\label{wave polynomials intro}%
\end{equation}
is an infinite system of solutions of the wave equation. Up to now, to our
best knowledge, no corresponding completeness result has been obtained. We
call the polynomials (\ref{wave polynomials intro}) and their finite linear
combinations \textbf{wave polynomials}, and one of the first results of the
present work is a Runge-type theorem establishing that any regular solution of
(\ref{wave intro}) in a closed square $\overline{R}$ with the vertices
$(\pm2b,0)$ and $(0,\pm2b)$, $b>0$ can be uniformly approximated on
$\overline{R}$ by the wave polynomials. This theorem is auxiliary for
obtaining a similar result for solutions of the Klein-Gordon equation with a
variable coefficient
\begin{equation}
u_{xx}-u_{tt}-q(x)u=0\label{KGintro}%
\end{equation}
which we consider next. The construction of an infinite system of solutions
similar to the wave polynomials was done in \cite{KRT} with the aid of L.~Bers' results on pseudoanalytic formal powers \cite{Berskniga} extended onto
the hyperbolic situation. Similarly to the wave polynomials these
\textbf{generalized wave polynomials }are components of formal powers,
solutions of a corresponding hyperbolic Vekua equation which locally behave as
powers of $z=x+jt$ but in general are not of course powers. Using recent
results from \cite{CKT} on mapping properties of transmutation operators we
show that the generalized wave polynomials are images of the wave polynomials
under the action of a transmutation operator. Due to the uniform boundedness
of the transmutation operator and of its inverse several useful properties of
the wave polynomials are preserved also in the case of their generalizations.
In particular, the expansion theorem and the Runge-type theorem result to be
valid.

All these observations lead to a new representation for the solution of the
Cauchy problem for (\ref{KGintro}). It is based on the expansion of the Cauchy
data into series in terms of a certain system of functions $\left\{
\varphi_{k}\right\}  _{k=0}^{\infty}$ which are introduced as recursive
integrals and arise in the spectral parameter power series (SPPS)
representation for solutions of Sturm-Liouville equations \cite{KrCV08},
\cite{KrPorter2010}. In \cite{KrCMA2011} a completeness of $\left\{
\varphi_{k}\right\}  _{k=0}^{\infty}$ in $L_{2}$ was proved. In \cite{KMoT}
this result was obtained for the space of continuous and piecewise
continuously differentiable functions. Here we show that the completeness of
$\left\{  \varphi_{k}\right\}  _{k=0}^{\infty}$ in the space of continuous
functions directly follows from the mapping properties of the transmutation operator and the Weierstrass approximation theorem. In \cite{KMoT} it was shown that several classical results from the
theory of power series can be generalized onto the series in terms of the
functions $\varphi_{k}$, including the Taylor formula. Here we present several
new results on the approximation of continuous functions by linear
combinations of functions $\varphi_{k}$. In particular, we show that the system of functions $\{\varphi_k\}_{k=0}^\infty$ in a real-valued case is a Tchebyshev system, prove a direct and an inverse approximation theorems and study algorithms for such approximation.

Using the results on the approximation by functions $\varphi_{k}$ we propose a
numerical method for solving the Cauchy problem for (\ref{KGintro}) and
illustrate its performance with several test examples. Once the Cauchy data
are approximated by functions $\varphi_{k}$, the approximate solution of the
Cauchy problem is written in a closed form. As for $t>0$ the approximate
solution is an exact solution of equation (\ref{KGintro}) the only task
consists in a good approximation of the Cauchy data. We show that in fact with
relatively few functions $\varphi_{k}$ involved, a remarkable accuracy is achieved.

\section{Wave polynomials}

Let us consider the wave equation
\begin{equation}
\square w=0,\qquad\square:=\frac{\partial^{2}}{\partial x^{2}}-\frac
{\partial^{2}}{\partial t^{2}} \label{wave equation}%
\end{equation}
and the following infinite family of its polynomial solutions%
\begin{equation}
\big\{ \operatorname{Re}(x+jt)^{n}\quad\text{and}\quad\operatorname{Im}%
(x+jt)^{n}\big\} _{n=0}^{\infty} \label{wave polynomials}%
\end{equation}
where $j$ is a hyperbolic imaginary unit, $j^{2}=1$.

It is easy to see that
\begin{equation}
\operatorname{Re}(x+jt)^{n}=\frac{1}{2}\big( (x+t)^{n}+(x-t)^{n}%
\big) \quad\text{and}\quad\operatorname{Im}(x+jt)^{n}=\frac{1}{2}%
\big( (x+t)^{n}-(x-t)^{n}\big) . \label{Re and Im}%
\end{equation}
Let us reorder these polynomials as follows
\[
p_{0}(x,t)=1\text{,\quad}p_{1}(x,t)=\operatorname{Re}(x+jt)=x\text{,\quad
}p_{2}(x,t)=\operatorname{Im}(x+jt)=t\text{,}%
\]%
\[
p_{3}(x,t)=\operatorname{Re}(x+jt)^{2}=x^{2}+t^{2}\text{,\quad}p_{4}%
(x,t)=\operatorname{Im}(x+jt)^{2}=2xt\text{,\ldots.}%
\]
The obtained family of solutions of (\ref{wave equation}) will be called
\textbf{wave polynomials}. It is convenient to write them also in the
following form%
\begin{equation}
p_{0}(x,t)=1\text{,\qquad}p_{m}(x,t)=
\begin{cases}
{\displaystyle \sum_{\text{even }k=0}^{\frac{m+1}{2}}\binom{\frac{m+1}{2}}{k}
x^{\frac{m+1}{2}-k}t^{k},} & m\text{ odd},\\
{\displaystyle \sum_{\text{odd }k=1}^{\frac{m}{2}}\binom{\frac{m}{2}}{k}
x^{\frac{m}{2}-k}t^{k},} & m\text{ even}.
\end{cases}
\label{pm}%
\end{equation}

Consider equation (\ref{wave equation}) together with the following Goursat
conditions%
\[
w=\varphi(x)\text{ for }x-t=0\text{ and }w=\psi(x)\text{ for }x+t=0
\quad(-b\leq x\leq b),
\]
assuming additionally that $\varphi(0)=\psi(0)$. It is well known (see, e.g.,
\cite[4.1.1-9.]{Polyanin}) that for $\varphi$ and $\psi$ belonging to
$C^{2}[-b,b]$ the solution of the Goursat problem exists, is unique and has
the form
\begin{equation}
w(x,t)=\varphi\Big(\frac{x+t}{2}\Big)+\psi\Big(\frac{x-t}{2}\Big)-\varphi(0).
\label{solution Goursat}%
\end{equation}
Its domain of definition is a closed square $\overline{R}$ with the vertices
$(\pm2b,0)$ and $(0,\pm2b)$.

\begin{proposition}
\label{Prop Solution Goursat real analytic}Let the boundary data $\varphi$ and
$\psi$ be real-analytic functions with the corresponding power series
expansions%
\begin{equation}
\varphi(x)=\sum_{n=0}^{\infty}\alpha_{n}x^{n}\quad\text{and}\quad\psi
(x)=\sum_{n=0}^{\infty}\beta_{n}x^{n},\label{powerseries}%
\end{equation}
uniformly convergent on $[-b,b]$ and satisfying necessary condition
$\varphi(0)=\psi(0)$, i.e., $\alpha_{0}=\beta_{0}$. Then the unique solution
of the Goursat problem has the form
\[
w(x,t)=\alpha_{0}p_{0}(x,t)+\sum_{n=1}^{\infty}\frac{\alpha_{n}+\beta_{n}%
}{2^{n}}p_{2n-1}(x,t)+\sum_{n=1}^{\infty}\frac{\alpha_{n}-\beta_{n}}{2^{n}%
}p_{2n}(x,t)
\]
where the series converge uniformly in $\overline{R}$.
\end{proposition}

\begin{proof}
From (\ref{Re and Im}) we have
\[
p_{2n-1}(x,t)=\frac{1}{2}\big(  (x+t)^{n}+(x-t)^{n}\big)  \quad\text{and}\quad
p_{2n}(x,t)=\frac{1}{2}\big(  (x+t)^{n}-(x-t)^{n}\big)
\]
and hence
\begin{equation}
(x+t)^{n}=p_{2n-1}(x,t)+p_{2n}(x,t)\quad\text{and}\quad(x-t)^{n}%
=p_{2n-1}(x,t)-p_{2n}(x,t)\text{,\qquad}n=1,2,\ldots. \label{x+t}%
\end{equation}
From (\ref{solution Goursat}) we obtain that the solution of the Goursat
problem has the form
\[
w(x,t)=\alpha_{0}+\sum_{n=1}^{\infty}\alpha_{n}\frac{(x+t)^{n}}{2^{n}}%
+\sum_{n=1}^{\infty}\beta_{n}\frac{(x-t)^{n}}{2^{n}}.
\]
Substitution of the relations (\ref{x+t}) gives us the equalities%
\begin{equation*}
\begin{split}
w(x,t) & =\alpha_{0}+\sum_{n=1}^{\infty}\alpha_{n}\frac{p_{2n-1}(x,t)+p_{2n}%
(x,t)}{2^{n}}+\sum_{n=1}^{\infty}\beta_{n}\frac{p_{2n-1}(x,t)-p_{2n}%
(x,t)}{2^{n}}\\
&=\alpha_{0}p_{0}(x,t)+\sum_{n=1}^{\infty}\frac{\alpha_{n}+\beta_{n}}{2^{n}%
}p_{2n-1}(x,t)+\sum_{n=1}^{\infty}\frac{\alpha_{n}-\beta_{n}}{2^{n}}%
p_{2n}(x,t).
\end{split}
\end{equation*}
\end{proof}

\begin{remark}
From this proposition we obtain that the wave polynomials represent a basis in
the linear space of solutions of the wave equation which admit a uniformly
convergent in $\overline{R}$ power series expansion with the center in the
origin. Indeed, consider any such solution of \eqref{wave equation} in
$\overline{R}$. Its values on the lines $x-t=0$ and $x+t=0$ admit uniformly
convergent power series expansion of the form \eqref{powerseries}. According
to the proposition the considered solution can be represented as a uniformly
convergent series of the wave polynomials.
\end{remark}

Let us prove the completeness of the wave polynomials in the linear space of
regular solutions of the wave equation with respect to the maximum norm.

\begin{theorem}
\label{Th Completeness wave polynomials}Let $w\in C^{2}(\overline{R})$ be a
solution of the wave equation \eqref{wave equation} in $R$. Then there exists
a sequence of wave polynomials $P_{N}=\sum_{n=0}^{N}a_{n}p_{n}$ uniformly
convergent to $w$ in $\overline{R}$.
\end{theorem}

\begin{proof}
We need to prove that for any $\varepsilon>0$ there exist such a number $N$
and coefficients $a_{n}$, $n=0,1,\ldots N$ that $\left\vert w(x,t)-P_{N}%
(x,t)\right\vert <\varepsilon$ for any point $(x,t)\in\overline{R}$. Let
$w=\varphi(x)$ for $x-t=0$ and $w=\psi(x)$ for $x+t=0$ ($-b\leq x\leq b$). We
choose $\varepsilon>0$ and such $\varepsilon_{1,2}>0$ that $\varepsilon
=2\varepsilon_{1}+2\varepsilon_{2}$. According to the Weierstrass theorem
there exists such number $N$ and such polynomials $p_{1}$ and $p_{2}$ of order
not greater than $N$ that $\left\vert \varphi(x)-p_{1}(x)\right\vert
<\varepsilon_{1}$ and $\left\vert \psi(x)-p_{2}(x)\right\vert <\varepsilon
_{2}$ ($-b\leq x\leq b$). We consider polynomials $q_{1}(x)=p_{1}%
(x)-p_{1}(0)+\varphi(0)$ and $q_{2}(x)=p_{2}(x)-p_{2}(0)+\psi(0)$ satisfying
the condition $q_{1}(0)=q_{2}(0)=\varphi(0)$. Due to Proposition
\ref{Prop Solution Goursat real analytic} the unique solution $\widetilde{w}$
of the Goursat problem with the boundary data $q_{1}$ and $q_{2}$ has the form
$\widetilde{w}=P_{N}(x,t)$ where $P_{N}(x,t)=q_{1}\big(\frac{x+t}%
{2}\big)+q_{2}\big(\frac{x-t}{2}\big)-q_{1}(0)$. Consider
\begin{equation*}
\begin{split}
\left\vert w(x,t)-\widetilde{w}(x,t)\right\vert &=\left\vert w(x,t)-P_{N}
(x,t)\right\vert \\
&\leq\Bigl\vert \varphi\Bigl(\frac{x+t}{2}\Bigr)-q_{1}\Bigl(\frac{x+t}
{2}\Bigr)\Bigr\vert +\Bigl\vert \psi\Bigl(\frac{x-t}{2}\Bigr)-q_{2}%
\Bigl(\frac{x-t}{2}\Bigr)\Bigr\vert\\
&\leq\Bigl\vert \varphi\Bigl(\frac{x+t}{2}\Bigr)-p_{1}\Bigl(\frac{x+t}
{2}\Bigr)\Bigr\vert+\bigl\vert\varphi(0)-p_{1}(0)\bigr\vert\\
&+\Bigl\vert \psi\Bigl(\frac{x-t}{2}\Bigr)-p_{2}\Bigl(\frac{x-t}{2}%
\Bigr)\Bigr\vert+\bigl\vert\psi(0)-p_{2}(0)\bigr\vert \leq2\varepsilon
_{1}+2\varepsilon_{2}=\varepsilon.
\end{split}
\end{equation*}
\end{proof}

\section{Transmutation operators and their action on powers of the independent
variable}

\subsection{Systems of recursive integrals}

Let $f\in C^{2}(a,b)\cap C^{1}[a,b]$ be a complex valued function and
$f(x)\neq0$ for any $x\in\lbrack a,b]$. The interval $(a,b)$ is supposed to be
finite. Let us consider the following auxiliary functions%
\begin{align}\label{X1}
\widetilde{X}^{(0)}(x) &\equiv X^{(0)}(x)\equiv1, \\
\widetilde{X}^{(n)}(x) & =n \int_{x_{0}}^{x} \widetilde{X}^{(n-1)}(s)\left(
f^{2}(s)\right)  ^{(-1)^{n-1}}\,\mathrm{d}s, \label{X2} \\
X^{(n)}(x)& =n \int_{x_{0}}^{x} X^{(n-1)}(s)\left(  f^{2}(s)\right)  ^{(-1)^{n}%
}\,\mathrm{d}s, \label{X3}
\end{align}
where $x_{0}$ is an arbitrary fixed point in $[a,b]$. We introduce the
infinite system of functions $\left\{  \varphi_{k}\right\}  _{k=0}^{\infty}$
defined as follows%
\begin{equation}
\varphi_{k}(x)=%
\begin{cases}
f(x)X^{(k)}(x), & k\text{\ odd,}\\
f(x)\widetilde{X}^{(k)}(x), & k\text{\ even,}%
\end{cases}
\label{phik}%
\end{equation}
where the definition of $X^{(k)}$ and $\widetilde{X}^{(k)}$ is given by
(\ref{X1})--(\ref{X3}) with $x_{0}$ being an arbitrary point of the interval
$[a,b]$.

\begin{example}
\label{ExamplePoly}Let $f\equiv1$, $a=0$, $b=1$. Then it is easy to see that
choosing $x_{0}=0$ we have $\varphi_{k}(x)=x^{k}$, $k\in\mathbb{N}_{0}$ where
by $\mathbb{N}_{0}$ we denote the set of non-negative integers.
\end{example}

In \cite{KrCMA2011} it was shown that the system $\left\{  \varphi
_{k}\right\}  _{k=0}^{\infty}$ is complete in $L_{2}(a,b)$ and in \cite{KMoT}
its completeness in the space of continuous and piecewise continuously
differentiable functions with respect to the maximum norm was obtained and the
corresponding series expansions in terms of the functions $\varphi_{k}$ were
studied. The completeness in the space $C[a,b]$ is shown in the Proposition
\ref{PropApproximabilityByPhi}.

The system (\ref{phik}) is closely related to the notion of the $L$-basis
introduced and studied in \cite{Fage}. Here the letter $L$ corresponds to a
linear ordinary differential operator. This becomes more transparent from the
following result obtained in \cite{KrCV08} (for additional details and simpler
proof see \cite{APFT} and \cite{KrPorter2010}) establishing the relation of
the system of functions $\left\{  \varphi_{k}\right\}  _{k=0}^{\infty}$ to
Sturm-Liouville equations.

\begin{theorem}
[\cite{KrCV08}]\label{ThGenSolSturmLiouville copy(1)} Let $q$ be a continuous
complex valued function of an independent real variable $x\in\lbrack a,b],$
$\lambda$ be an arbitrary complex number. Suppose there exists a solution $f$
of the equation
\begin{equation}
f^{\prime\prime}-qf=0 \label{SLhom}%
\end{equation}
on $(a,b)$ such that $f\in C^{2}[a,b]$ and $f\neq0$ on $[a,b]$. Then the
general solution of the equation
\begin{equation}
u^{\prime\prime}-qu=\lambda u
\end{equation}
on $(a,b)$ has the form%
\[
u=c_{1}u_{1}+c_{2}u_{2}%
\]
where $c_{1}$ and $c_{2}$ are arbitrary complex constants,
\begin{equation}
u_{1}=%
{\displaystyle\sum\limits_{k=0}^{\infty}}
\frac{\lambda^{k}}{(2k)!}\varphi_{2k}\quad\quad\text{and}\quad\quad u_{2}=%
{\displaystyle\sum\limits_{k=0}^{\infty}}
\frac{\lambda^{k}}{(2k+1)!}\varphi_{2k+1} \label{u1u2}%
\end{equation}
and both series converge uniformly on $[a,b]$.

The solutions $u_{1}$ and $u_{2}$ satisfy the initial conditions
\begin{align*}
u_{1}(x_{0})  &  =f(x_{0}), & u_{1}^{\prime}(x_{0})  &  =f^{\prime}(x_{0}),\\
u_{2}(x_{0})  &  =0, & u_{2}^{\prime}(x_{0})  &  =1/f(x_{0}).
\end{align*}

\end{theorem}

Together with the family of functions $\left\{  \varphi_{k}\right\}
_{k=0}^{\infty}$ we consider a dual system of recursive integrals defined by
the following relations involving the \textquotedblleft second
half\textquotedblright\ of the formal powers \eqref{X1}--\eqref{X3},
\begin{equation}
\psi_{k}(x)=%
\begin{cases}
\dfrac{\widetilde{X}^{(k)}(x)}{f(x)}, & k\text{\ odd,}\\
\dfrac{X^{(k)}(x)}{f(x)}, & k\text{\ even.}%
\end{cases}
\label{psik}%
\end{equation}

\subsection{Generalized derivatives and generalized Taylor
series\label{Subsect Gen Der and Gen Taylor}}

In \cite{KMoT} a notion of the generalized derivative was introduced which
alows one to extend the theory of power series onto the series in terms of the
functions $\varphi_{k}$ (the formal power series). Here we slightly modify the
definition introduced in \cite{KMoT}. This modification simplifies formulas
involving the generalized derivatives and reflects a better understanding of
the nature of the functions $\varphi_{k}$ and $\psi_{k}$ in the light of
application of transmutation operators. We assume that the complex-valued
function $f$ is continuous on $[a,b]$, $f(x)\neq0$ for any $x\in\lbrack a,b]$
and $f(x_{0})=1$.

\begin{definition}
\label{Def f-derivatives}The generalized derivatives or the $f$-derivatives of
a function $g$ are defined by the following relations whenever they make
sense. The generalized derivative of a zero order coincides with the function
$g$, $d_{0}^{f}[g](x)=g(x)$. The generalized derivatives of higher orders are
defined as follows $d_{k}^{f}[g]=f^{(-1)^{k-1}}\frac{d}{dx}\left(
f^{(-1)^{k}}d_{k-1}^{f}[g]\right)  $, $k=1,2,\ldots$.

That is,
\[
d_{k}^{f}[g]=%
\begin{cases}
f\frac{d}{dx}\left(  \frac{1}{f}d_{k-1}^{f}[g]\right)  , & k\ \text{ odd},\\
\frac{1}{f}\frac{d}{dx}\left(  fd_{k-1}^{f}[g]\right)  , & k\ \text{ even}.
\end{cases}
\]

\end{definition}

\begin{remark}
\label{Rem Second derivative and Factorization}Let $f$ be a solution of
\eqref{SLhom} satisfying the conditions of Theorem
\ref{ThGenSolSturmLiouville copy(1)}. Then the corresponding differential
operator can be factorized in the following way $L=\frac{d^{2}}{dx^{2}}-q(x)$
$=\frac{1}{f}\frac{d}{dx}\left(  f^{2}\frac{d}{dx}\frac{1}{f}\ \cdot\right)
$. This factorization sometimes is called the Polya factorization (see
\cite{KelleyPeterson}).We see from it that $L=d_{2}^{f}$.

The generalized derivative $d_{1}^{f}=$ $f\frac{d}{dx}\left(  \frac{1}%
{f}\ \cdot\right)  $ coincides with the Darboux transformation (see, e.g.,
\cite{Matveev}).
\end{remark}

\begin{remark}
\label{Rem Derivatives phi}It is easy to see that
\[
d_{1}^{f}\varphi_{k}=k\psi_{k-1},\quad k=1,2,\ldots,
\]
\[
d_{2}^{f}\varphi_{k}=k(k-1)\varphi_{k-2},\quad k=2,3,\ldots.
\]
and
\[
d_{1}^{f}\varphi_{0}=d_{2}^{f}\varphi_{0}=d_{2}^{f}\varphi_{1}=0.
\]

\end{remark}

\begin{remark}
Consideration of the $1/f$-derivatives defined according to Definition
\ref{Def f-derivatives} leads to the dual relations
\[
d_{1}^{1/f}\psi_{k}=k\varphi_{k-1},\quad k=1,2,\ldots
\]
and
\[
d_{2}^{1/f}=f\frac{d}{dx}\left(  \frac{1}{f^{2}}\frac{d}{dx}f\ \cdot\right)
=\frac{d^{2}}{dx^{2}}-q_{D}(x)
\]
where the potential $q_{D}$ is a superpartner of $q$ defined by the equality
$q_{D}=-q+2\left(  \frac{f^{\prime}}{f}\right)  ^{2}$ (see Subsection
\ref{Subsect Transmutations}).
\end{remark}

\begin{definition}
\label{Defn Pn} Functions of the form
\begin{equation}
P_{n}(x)=\sum\limits_{k=0}^{n}\alpha_{k}\varphi_{k}(x) \label{Pn}%
\end{equation}
where $\alpha_{k}$, $k=0,1,\ldots,n$ are complex numbers are called
$f$-polynomials of the order $n$.
\end{definition}

In a complete similarity to the fact that the coefficients of a polynomial
$\sum_{k=0}^{n}a_{k}(x-x_{0})^{k}$ can be expressed through its value and the
values of its derivatives at the point $x_{0}$, the coefficients of an
$f$-polynomial are determined by the values of $P_{n}$ and of its
$f$-derivatives at $x_{0}$ (at the initial point of integration in (\ref{X2}),
(\ref{X3})). Indeed, a simple calculation using Remark
\ref{Rem Derivatives phi} gives us the following result
\[
\alpha_{k}=\frac{d_{k}^{f}[P_{n}](x_{0})}{k!}.
\]
Let us consider a function $g$ possessing at the point $x_{0}$ the
$f$-derivatives of all orders up to the order $n$. More precisely this means
that the function $g$ is defined and possesses the $f$-derivatives of all
orders up to the order $n-1$ in some segment $[a,b]$ containing the point
$x_{0}$ and additionally there exists the $n$-th $f$-derivative of $g$ at the
point $x_{0}$. In relation with the function $g$, we introduce an
$f$-polynomial of the form (\ref{Pn}) with the coefficients
\[
\alpha_{k}=\frac{d_{k}^{f}[g](x_{0})}{k!}.
\]
According to the previous observation, this $f$-polynomial together with its
$f$-derivatives at $x_{0}$ up to the order $n$ take the same values as the
function $g$ and its respective $f$-derivatives, $d_{k}^{f}[P_{n}%
](x_{0})=d_{k}^{f}[g](x_{0})$, $k=0,1,\ldots,n$. We are interested in
estimating the difference between $P_{n}(x)$ and $g(x)$ for $x\neq x_{0}$.

\begin{theorem}
[Generalized Taylor theorem with the Peano form of the remainder term]Let the
function $g$ possesses at the point $x_{0}$ the $f$-derivatives of all orders
up to the order $n$ and $f$ be a continuously differentiable function in a
neighborhood of $x_{0}$. Then
\[
g(x)=\sum_{k=0}^{n}\frac{d_{k}^{f}[g](x_{0})}{k!}\varphi_{k}(x)+o\big((x-x_{0}%
)^{n}\big).
\]

\end{theorem}

\begin{proof}
Consider the difference $r(x)=g(x)-P_{n}(x)$. We have
\begin{equation}
r(x_{0})=d_{1}^{f}[r](x_{0})=\cdots=d_{n}^{f}[r](x_{0})=0.
\label{conditions for r}%
\end{equation}
Let us prove by induction that if a function $r$ satisfies the conditions
(\ref{conditions for r}) or the conditions
\begin{equation}
r(x_{0})=d_{1}^{1/f}[r](x_{0})=\cdots=d_{n}^{1/f}[r](x_{0})=0,
\label{conditions for d1}%
\end{equation}
then necessarily $r(x)=o\big((x-x_{0})^{n}\big)$.

For $n=1$ this assertion has the form: if the function $r(x)$ possessing at
$x_{0}$ the first $f$-derivative fulfills the conditions $r(x_{0})=d_{1}%
^{f}[r](x_{0})=0$ or possessing at $x_{0}$ the first $1/f$-derivative fulfills
the conditions $r(x_{0})=d_{1}^{1/f}[r](x_{0})=0$ then $r(x)=o(x-x_{0})$. Its
validity can be verified directly. In the first case we have
\[
\lim_{x\rightarrow x_{0}}\frac{r(x)}{x-x_{0}}=f(x_{0})\lim_{x\rightarrow
x_{0}}\frac{r(x)/f(x)}{x-x_{0}}=f(x_{0})\lim_{x\rightarrow x_{0}}\frac
{\frac{r(x)}{f(x)}-\frac{r(x_{0})}{f(x_{0})}}{x-x_{0}}=0
\]
due to the condition $d_{1}^{f}[r](x_{0})=0$, and in the second case the proof
is completely similar.

Assume that the assertion is true for some $n\geq1$. Due to the symmetry of
(\ref{conditions for r}) and (\ref{conditions for d1}) it is enough to prove
that if for a function $r(x)$ possessing at $x_{0}$ the $f$-derivatives up to
the order $n+1$ the following conditions are fulfilled $r(x_{0})=d_{1}%
^{f}[r](x_{0})=\cdots=d_{n+1}^{f}[r](x_{0})=0$ then $r(x)=o\big((x-x_{0}%
)^{n+1}\big)$. For this we observe that $r(x)$ fulfills the conditions
(\ref{conditions for r}) meanwhile $d_{1}^{f}[r]$ fulfills the conditions
(\ref{conditions for d1}) and hence by the assumption we have
$r(x)=o\big((x-x_{0})^{n}\big)$ and $d_{1}^{f}[r](x)=o\big((x-x_{0}%
)^{n}\big).$ Notice that by the mean value theorem
\begin{align*}
r(x) =  &  r(x)-r(x_{0})=\left(  \operatorname{Re}r^{\prime}(c_{1}%
)+i\operatorname{Im}r^{\prime}(c_{2})\right)  (x-x_{0})\\
=  &  \left(  \operatorname{Re}\Bigl( d_{1}^{f}[r](c_{1})+\frac{f^{\prime
}(c_{1})}{f(c_{1})}r(c_{1})\Bigr) +i\operatorname{Im}\Bigl( d_{1}^{f}%
[r](c_{2})+\frac{f^{\prime}(c_{2})}{f(c_{2})}r(c_{2})\Bigr) \right)
(x-x_{0}),
\end{align*}
where $c_{1}$ and $c_{2}$ are located between $x_{0}$ and $x$. As $\left\vert
c_{1,2}-x_{0}\right\vert <\left\vert x-x_{0}\right\vert $, then $d_{1}%
^{f}[r](c_{1,2})=o\bigl((c_{1,2}-x_{0})^{n}\bigr)=o\bigl((x-x_{0})^{n}\bigr)$
and $r(c_{1,2})=o\bigl((c_{1,2}-x_{0})^{n}\bigr)=o\bigl((x-x_{0})^{n}\bigr)$.
Thus, we obtain $r(x)=o\bigl((x-x_{0})^{n+1}\bigr)$.
\end{proof}

Under an additional condition that $f$ is real valued we obtain the following
result by applying the reasoning from \cite{KMoT}.

\begin{theorem}
[Generalized Taylor theorem with the Lagrange form of the remainder]%
\label{ThGenTaylorTheorem} Let the real-valued function $g$ possesses on the
segment $[x_{0},b]$ continuous $f$-derivatives of all orders up to the order
$n$ and there exists a bounded $(n+1)$-th $f$-derivative of $g$ on $\left(
x_{0},b\right)  $. Let $f$ be a real-valued, continuously differentiable
function in $[x_{0},b]$. Then for any $x\in\lbrack x_{0},b]$ there exists a
number $c$ between $x_{0}$ and $x$ such that
\[
g(x)=\sum_{k=0}^{n}\frac{d_{k}^{f}(g)(x_{0})}{k!}\varphi_{k}(x)+\frac
{d_{n+1}^{f}(g)(c)}{(n+1)!}\varphi_{n+1}(x).
\]

\end{theorem}

\begin{proof}
The proof is a simple adaptation of the proof from \cite{KMoT} according to
the modified definition of generalized derivatives. All the steps and
reasonings do not essentially change.
\end{proof}

Obviously, the classical Taylor theorem with the Lagrange form of the
remainder term is a special case of theorem \ref{ThGenTaylorTheorem} when
$f\equiv1$.

\subsection{Transmutation operators\label{Subsect Transmutations}}

We give a definition of a transmutation operator from \cite{KT Obzor} which is
a modification of the definition given by Levitan \cite{LevitanInverse}
adapted to the purposes of the present work. Let $E$ be a linear topological
space and $E_{1}$ its linear subspace (not necessarily closed). Let $A$ and
$B$ be linear operators: $E_{1}\rightarrow E$.

\begin{definition}
\label{DefTransmut} A linear invertible operator $T$ defined on the whole $E$
such that $E_{1}$ is invariant under the action of $T$ is called a
transmutation operator for the pair of operators $A$ and $B$ if it fulfills
the following two conditions.

\begin{enumerate}
\item Both the operator $T$ and its inverse $T^{-1}$ are continuous in $E$;

\item The following operator equality is valid
\begin{equation}
AT=TB \label{ATTB}%
\end{equation}
or which is the same
\[
A=TBT^{-1}.
\]

\end{enumerate}
\end{definition}

Very often in literature the transmutation operators are called the
transformation operators. Here we keep the original term introduced by
Delsarte and Lions \cite{DelsarteLions1956}. Our main interest concerns the
situation when $A=-\frac{d^{2}}{dx^{2}}+q(x)$, $B=-\frac{d^{2}}{dx^{2}}$,
\ and $q$ is a continuous complex-valued function. Hence for our purposes it
will be sufficient to consider the functional space $E=C[a,b]$ with the
topology of uniform convergence and its subspace $E_{1}$ consisting of
functions from $C^{2}\left[  a,b\right]  $. One of the possibilities to
introduce a transmutation operator on $E$ was considered by Lions
\cite{Lions57} and later on in other references (see, e.g., \cite{Marchenko}),
and consists in constructing a Volterra integral operator corresponding to a
midpoint of the segment of interest. As we begin with this transmutation
operator it is convenient to consider a symmetric segment $[-b,b]$ and hence
the functional space $E=C[-b,b]$. It is worth mentioning that other well known
ways to construct the transmutation operators (see, e.g.,
\cite{LevitanInverse}, \cite{Trimeche}) imply imposing initial conditions on
the functions and consequently lead to transmutation operators satisfying
(\ref{ATTB}) only on subclasses of $E_{1}$.

Thus, we consider the space $E=C[-b,b]$ and an operator of transmutation for
the defined above $A$ and $B$ can be realized in the form (see, e.g.,
\cite{LevitanInverse} and \cite{Marchenko}) of a Volterra integral operator%

\begin{equation}
Tu(x)=u(x)+\int_{-x}^{x}K(x,t)u(t)dt \label{T}%
\end{equation}
where $K(x,t)$ is a unique solution of the Goursat problem%

\begin{equation}
\left(  \frac{\partial^{2}}{\partial x^{2}}-q(x)\right)  K(x,t)=\frac
{\partial^{2}}{\partial t^{2}}K(x,t), \label{Goursat1}%
\end{equation}%
\begin{equation}
K(x,x)=\frac{1}{2}\int_{0}^{x}q(s)ds,\qquad K(x,-x)=0. \label{Goursat2}%
\end{equation}

In \cite{CKT} the following mapping properties of the operator $T$ were proved.

\begin{theorem}
[\cite{CKT}]\label{Th Transmutation of Powers K}Let $q$ be a continuous
complex valued function of an independent real variable $x\in\lbrack-b,b]$ for
which there exists a particular solution $f$ of \eqref{SLhom} such that $f\in
C^{2}[-b,b]$, $f\neq0$ on $[-b,b]$ and normalized as $f(0)=1$. Denote
$h:=f^{\prime}(0)\in\mathbb{C}$. Suppose $T$ is the operator defined by
\eqref{T} where the kernel $K$ is a solution of the problem \eqref{Goursat1},
\eqref{Goursat2} and $\varphi_{k}$, $k\in\mathbb{N}_{0}$ are functions defined
by \eqref{phik}. Then the following equalities hold%
\begin{equation}
\varphi_{k}=T[x^{k}],\quad\quad k\text{ odd} \label{Txk_odd}%
\end{equation}
and
\begin{equation}
\varphi_{k}-\frac{h}{k+1}\varphi_{k+1}=T[x^{k}],\quad\quad k\text{ even.}
\label{Txk_even1}%
\end{equation}

\end{theorem}

\begin{remark}
Let $f$ be the solution of \eqref{SLhom} satisfying the initial conditions
\begin{equation}
f(0)=1\quad\text{and}\quad f^{\prime}(0)=0. \label{initcond 1 0}%
\end{equation}
If it does not vanish on $[-b,b]$ then from Theorem
\ref{Th Transmutation of Powers K} we obtain that $\varphi_{k}=T[x^{k}]$ for
any $k\in\mathbb{N}_{0}$. In general, of course there is no guaranty that the
solution with such initial values would have no zero on $[-b,b]$, and hence
the operator $T$ transmutes the powers of $x$ into $\varphi_{k}$ whose
construction is based on the solution $f$ satisfying \eqref{initcond 1 0} only
in some neighborhood of the origin.
\end{remark}

In \cite{CKT} it was shown that given a system of functions $\left\{
\varphi_{k}\right\}  _{k=0}^{\infty}$ defined by (\ref{phik}) where $f$ is any
particular solution of (\ref{SLhom}) such that $f\in C^{2}[-b,b]$, $f\neq0$ on
$[-b,b]$ and $f(0)=1$, $f^{\prime}(0)=h\in\mathbb{C}$, the operator $T$ can be
modified in such a way that the new transmutation operator will map $x^{k}$ to
$\varphi_{k}$ for any $k\in\mathbb{N}_{0}$.

\begin{theorem}
[\cite{CKT}, \cite{KT}]\label{Th Transmute}Under the conditions of Theorem
\ref{Th Transmutation of Powers K} the operator
\begin{equation}
\mathbf{T}u(x)=u(x)+\int_{-x}^{x}\mathbf{K}(x,t;h)u(t)dt \label{Tmain}%
\end{equation}
with the kernel defined by
\begin{equation}
\mathbf{K}(x,t;h)=\frac{h}{2}+K(x,t)+\frac{h}{2}\int_{t}^{x}\left(
K(x,s)-K(x,-s)\right)  ds \label{Kmain}%
\end{equation}
transforms $x^{k}$ into $\varphi_{k}(x)$ for any $k\in\mathbb{N}_{0}$ and
\begin{equation}
\left(  -\frac{d^{2}}{dx^{2}}+q(x)\right)  \mathbf{T}[u]=\mathbf{T}\left[
-\frac{d^{2}}{dx^{2}}(u)\right]  \label{TransmutC2}%
\end{equation}
for any $u\in C^{2}[-b,b]$.

Moreover, if the potential $q\in C^{1}[-b,b]$, then the kernel $\mathbf{K}%
(x,t;h)$ is a unique solution of the Goursat problem
\begin{equation}
\left(  \frac{\partial^{2}}{\partial x^{2}}-q(x)\right)  \mathbf{K}%
(x,t;h)=\frac{\partial^{2}}{\partial t^{2}}\mathbf{K}(x,t;h),
\label{GoursatTk1}%
\end{equation}%
\begin{equation}
\mathbf{K}(x,x;h)=\frac{h}{2}+\frac{1}{2}\int_{0}^{x}q(s)\,ds,\qquad
\mathbf{K}(x,-x;h)=\frac{h}{2}. \label{GoursatTk2}%
\end{equation}

\end{theorem}

This theorem was proved in \cite{CKT} under an additional assumption that the
potential $q$ must be continuously differentiable, and in \cite{KT} it was
shown that this assumption was superfluous due to new observed relations
(\ref{CommutT1dx}), (\ref{CommutT2dx}) given below between the transmutations
for Darboux transformed Schr\"{o}dinger operators. The last statement of this
theorem was proved in \cite{KT}, also the interested reader may find in
\cite{KT} necessary changes regarding the case when $q\in C[-b,b]$. For
brevity, we omit these details in the present article.

In the following sections we use both the transmutation operator $\mathbf{T}$
and its inverse $\mathbf{T}^{-1}$, and the norms of these operators appear in
many estimates. Hence it is natural to obtain convenient estimates for the
norms. Remind that in \cite{KT} it was mentioned that to define the
transmutation operator $\mathbf{T}$, we need to know its integral kernel in
the domain $0\leq|t|\leq|x|\leq b$. But the Goursat problem
\eqref{GoursatTk1}--\eqref{GoursatTk2} is also well-posed and allows to
determine the kernel $\mathbf{K}(x,t;h)$ in the domain $0\leq|x|\leq|t|\leq
b$. Thus we may assume that the integral kernel $\mathbf{K}(x,t;h)$ is known
in the square $|x|\leq b$, $|t|\leq b$. In such case, there is a simple
representation of the inverse operator $\mathbf{T}^{-1}$.

\begin{theorem}
[\cite{KT}]\label{Th Inverse}The inverse operator $\mathbf{T}^{-1}$ can be
represented as the Volterra integral operator
\begin{equation}
\mathbf{T}^{-1}u(x)=u(x)-\int_{-x}^{x}\mathbf{K}(t,x;h)u(t)\,dt.
\label{Tinverse}%
\end{equation}

\end{theorem}

Both $\mathbf{T}$ and $\mathbf{T}^{-1}$ are obviously bounded as operators
from $C[-b,b]$ to itself. The estimates for their norms depend on the
estimates for the integral kernels, e.g., for $\left\Vert \mathbf{T}%
\right\Vert $ we have $\left\Vert \mathbf{T}\right\Vert \leq1+2b\max\left\vert
\mathbf{K}(x,t;h)\right\vert $. Some estimates for the integral kernel
$K(x,t)$ can be found in \cite{Marchenko}. From them corresponding estimates
for the kernel $\mathbf{K}(x,t;h)$ can be obtained using (\ref{Kmain}).
However, the growth with the increase of the interval of mentioned estimates
is immensely fast even for the simplest potentials. We adapt the general
method of successive approximations for solving Goursat problems (see, e.g.
\cite{Vladimirov}) to obtain better estimates for the kernel $\mathbf{K}%
(x,t;h)$.

\begin{proposition}
\label{PrKEstimate} Let $q$ be a continuous complex valued function of an
independent real variable $x\in[-b,b]$. Then the kernel $\mathbf{K}(x,t;h)$ in
the square $|x|\le b$, $|t|\le b$ satisfies the following estimate
\begin{equation}
\label{KEstimate}|\mathbf{K}(x,t;h)|\le\frac{|h|}{2} I_{0}\big(\sqrt
{c|x^{2}-t^{2}|}\big)+\frac12 \frac{\sqrt{c|x^{2}-t^{2}|}I_{1}\big(\sqrt
{c|x^{2}-t^{2}|}\big)}{|x-t|},
\end{equation}
where $c=\max_{[-b,b]}|q(x)|$ and $I_{0}$ and $I_{1}$ are modified Bessel
functions of the first kind.
\end{proposition}

\begin{remark}
Note that for the case of operator $\partial^{2}-c$ with constant potential
$c>0$ and $h>0$ the exact kernel of transmutation operator in the domain $0\le
t\le x\le b$ coincides with the right-hand side of \eqref{KEstimate}, see
\cite{CKT}.
\end{remark}

\begin{proof}
The proof follows the proof from \cite{Vladimirov}. First, we introduce the
function $H(u,v):=\mathbf{K}(u+v, u-v; h)$. It satisfies the Goursat problem
(see \cite{KT})
\begin{equation}
\frac{\partial^{2}H(u,v)}{\partial u\,\partial v}=q(u+v)H(u,v),
\label{GoursatTh1}%
\end{equation}
\begin{equation}
H(u,0)=\frac{h}{2}+\frac{1}{2}\int_{0}^{u}q(s)\,ds,\qquad H(0,v)=\frac{h}{2}
\label{GoursatTh2}%
\end{equation}
in the domain $|u|+|v|\le b$. It is worth mentioning that despite the kernel
$\mathbf{K}(x,t;h)$ is not the classical solution of the problem
\eqref{GoursatTk1}--\eqref{GoursatTk2} in the case when $q\in C[-b,b]$,
nevertheless the function $H(u,v)$ is a classical solution of the problem
\eqref{GoursatTh1}--\eqref{GoursatTh2}, see \cite{KT}. Define $G:=\frac
{\partial H}{\partial u}$. Then the Goursat problem
\eqref{GoursatTh1}--\eqref{GoursatTh2} is equivalent to the system of integral
equations
\[%
\begin{cases}
H(u,v)=\frac h2+\int_{0}^{u} G(u^{\prime}, v)\, du^{\prime}\\
G(u,v)=\frac12 q(u)+\int_{0}^{v} q(u+v^{\prime})H(u,v^{\prime})\, dv^{\prime}.
\end{cases}
\]
Applying the successive approximations method for this system, we obtain
\[
|H(u,v)|\le\frac{|h|}2\sum_{k=0}^{\infty}\frac{c^{k} |uv|^{k}}{k! k!}%
+\frac12\sum_{k=0}^{\infty}\frac{c^{k+1}|u|^{k+1}|v|^{k}}{(k+1)!k!},
\]
which coincides with \eqref{KEstimate}.
\end{proof}

Since the function $I_{1}(x)/x$ is monotone increasing for $x>0$, we obtain
\[
\frac{\sqrt{c|x^{2}-t^{2}|}I_{1}\big(\sqrt{c|x^{2}-t^{2}|}\big)}%
{|x-t|}=c|x+t|\frac{I_{1}\big(\sqrt{c|x^{2}-t^{2}|}\big)}{\sqrt{c|x^{2}%
-t^{2}|}}\le2bc\frac{I_{1}(b\sqrt{c})}{b\sqrt{c}}=2\sqrt{c}I_{1}(b\sqrt{c})
\]
for $|x|\le b$ and $|t|\le b$, and the following estimate for the norms of
transmutation operator and of its inverse immediately follows from Proposition
\ref{PrKEstimate}.

\begin{corollary}
The following estimate holds
\[
\max\big\{\|\mathbf{T}\|, \|\mathbf{T}^{-1}\|\big\} \le1+b\left(  |h|
I_{0}(b\sqrt{c}) + 2\sqrt{c} I_{1}(b\sqrt{c})\right)  ,
\]
where $c=\max_{[-b,b]}|q(x)|$ and $I_{0}$ and $I_{1}$ are modified Bessel
functions of the first kind.
\end{corollary}

Together with the operator $\frac{d^{2}}{dx^{2}}-q(x)$ let us consider a
Darboux associated operator $\frac{d^{2}}{dx^{2}}-q_{D}(x)$ with the potential
defined by the equality $q_{D}=-q+2\left(  \frac{f^{\prime}}{f}\right)  ^{2}$
where $f\in C^{2}[-b,b]$ is a solution of (\ref{SLhom}), $f\neq0$ on $[-b,b]$,
$f(0)=1$ and $h=f^{\prime}(0)\in\mathbb{C}$. In \cite{KT} explicit formulas
were obtained for the kernel $\mathbf{K}_{D}(x,t;-h)$ in terms of
$\mathbf{K}(x,t;h)$, where $\mathbf{K}_{D}(x,t;-h)$ is the integral kernel of
the transmutation operator $\mathbf{T}_{D}$ which satisfies the equality
\[
\left(  -\frac{d^{2}}{dx^{2}}+q_{D}(x)\right)  \mathbf{T}_{D}[u]=\mathbf{T}%
_{D}\left[  -\frac{d^{2}}{dx^{2}}(u)\right]
\]
for any $u\in C^{2}[-b,b]$ and transforms $x^{k}$ into the functions $\psi
_{k}(x)$, $k\in\mathbb{N}_{0}$ defined by the relations (\ref{psik}). Note
that $\psi_{0}$ is obviously a solution of $\left(  -\frac{d^{2}}{dx^{2}%
}+q_{D}(x)\right)  \psi_{0}=0$ with the initial values $\psi_{0}(0)=1$ and
$\psi_{0}^{\prime}(0)=-h$.

The operator $\mathbf{T}_{D}$ has the form \cite{KT}
\[
\mathbf{T}_{D}[u](x)=u(x)+\int_{-x}^{x}\mathbf{K}_{D}%
(x,t;-h)u(t)\,dt,\label{T2}%
\]
with the kernel
\[
\mathbf{K}_{D}(x,t;-h)=-\frac{1}{f(x)}\bigg(\int_{-t}^{x}\partial
_{t}\mathbf{K}(s,t;h)f(s)\,ds+\frac{h}{2}f(-t)\bigg)\label{K2},
\]
and the following operator equalities hold on $C^{1}[-b,b]$:
\begin{align}
\frac{d}{dx}f\mathbf{T}_{D}  &  =f\mathbf{T}\frac{d}{dx}\label{CommutT1dx}\\
\frac{d}{dx}\frac{1}{f}\mathbf{T}  &  =\frac{1}{f}\mathbf{T}_{D}\frac{d}{dx}.
\label{CommutT2dx}%
\end{align}
These commutation equalities involving the operators of transmutation and
derivatives together with the property of the transmutation operators that if $u\in C^1[-b,b]$ then $\mathbf{T}^{-1} u \in C^1[-b,b]$, see Theorem \ref{Th Inverse},  lead to the following useful statement.

\begin{proposition}
[\cite{KT Obzor}]\label{GeneralDerivTransm} Let $u\in C^{n}[-b,b]$ and
$g=\mathbf{T}u$. Then there exist the first $n$ $f$-derivatives of $g$ on
$[-b,b]$, and the following equalities hold for $0\leq k\leq n$
\[
d_{k}^{f}(g)=\mathbf{T}_{D}u^{(k)},\qquad k\text{ odd},
\]
and
\[
d_{k}^{f}(g)=\mathbf{T}u^{(k)},\qquad k\text{ even}.
\]
The inverse statement, i.e., if there exist the first $n$ $f$-derivatives of
$g$ on $[-b,b]$, then $u=\mathbf{T}^{-1}g\in C^{n}[-b,b]$ is also true.
\end{proposition}

\section{Generalized wave polynomials}

Let us consider the following Klein-Gordon equation with a position dependent
mass%
\begin{equation}
\left(  \frac{\partial^{2}}{\partial x^{2}}-q(x)\right)  u(x,t)=\frac
{\partial^{2}}{\partial t^{2}}u(x,t)\label{KG}%
\end{equation}
where we assume that $q:[-b,b]\rightarrow\mathbb{C}$ and $q\in C[-b,b]$.
Suppose there exists a particular solution $f$ of equation (\ref{SLhom}) such
that $f\in C^{2}[-b,b]$ and $f\neq0$ on $[-b,b]$. We normalize it as $f(0)=1$
and set $h:=f^{\prime}(0)$.

Consider the system of functions $\left\{  \varphi_{k}\right\}  _{k=0}%
^{\infty}$ defined by (\ref{phik}) with $x_{0}=0$. Then due to Theorem
\ref{Th Transmute}, $\varphi_{k}(x)=$ $\mathbf{T}x^{k}$ for any $k\in
\mathbb{N}_{0}$ and due to (\ref{pm}) we obtain that the functions
\begin{equation}
u_{0}=f(x)\text{,\qquad}u_{m}(x,t)=%
\begin{cases}
{\displaystyle\sum_{\text{even }k=0}^{\frac{m+1}{2}}\binom{\frac{m+1}{2}}%
{k}\varphi_{\frac{m+1}{2}-k}(x)t^{k},} & m\text{ odd},\\
{\displaystyle\sum_{\text{odd }k=1}^{\frac{m}{2}}\binom{\frac{m}{2}}{k}%
\varphi_{\frac{m}{2}-k}(x)t^{k},} & m\text{ even},
\end{cases}
\label{um}%
\end{equation}
are solutions of (\ref{KG}) for any $-b<x<b$ and $-\infty<t<\infty$. Indeed,
we have that
\begin{equation}
u_{m}=\mathbf{T}p_{m}\quad\text{for every }m\in\mathbb{N}_{0}\text{.}%
\label{umTpm}%
\end{equation}
Moreover, the functions $u_{m}$ arise also as scalar (real, when $f$ is real
valued) parts of hyperbolic pseudoanalytic formal powers corresponding to the
generating pair $(f,j/f)$ where $\text{j}$ is a hyperbolic imaginary unit,
$j^{2}=1$ (see \cite{KRT}, \cite{APFT}).

Equalities (\ref{umTpm}) together with the completeness of the wave
polynomials (Theorem \ref{Th Completeness wave polynomials}) and the
boundedness of $\mathbf{T}$ and $\mathbf{T}^{-1}$ imply the completeness of
the \textbf{generalized wave polynomials }$u_{m}$ in the linear space of
regular solutions of (\ref{KG}).

\begin{theorem}
Let $u\in C^{2}(\overline{\mathbf{R}})$ be a solution of \eqref{KG} in
$\mathbf{R}$ where $\mathbf{R}$ is a square with the vertices $(\pm b,0)$ and
$(0,\pm b)$. Then there exists a sequence of generalized wave polynomials
$U_{N}=\sum_{n=0}^{N}a_{n}u_{n}$ uniformly convergent to $u$ in $\overline
{\mathbf{R}}$.
\end{theorem}

\begin{proof}
We have that $u=$ $\mathbf{T}w$ where $w$ is a $C^{2}$-solution of
(\ref{wave equation}) and due to Theorem
\ref{Th Completeness wave polynomials} for any $\varepsilon_{1}>0$ there
exists a wave polynomial $P_{N}$ such that $\max_{\overline{\mathbf{R}}%
}\left\vert w-P_{N}\right\vert <\varepsilon_{1}$. Thus, $\max_{\overline
{\mathbf{R}}}\left\vert u-\mathbf{T}P_{N}\right\vert =\max_{\overline
{\mathbf{R}}}\left\vert \mathbf{T}w-\mathbf{T}P_{N}\right\vert \leq
\varepsilon_{1}C=\varepsilon$. Here the constant $C$ depends only on the
kernel $\mathbf{K}(x,t;h)$.
\end{proof}

\begin{remark}
\label{Rem u at 0}When $t=0$ the following relations are valid%
\[
u_{m}(x,0)=
\begin{cases}
\varphi_{\frac{m+1}{2}}(x), & m\text{ odd}\\
0, & m\text{ even}%
\end{cases}
\]
and
\[
\frac{\partial u_{m}(x,0)}{\partial t}=%
\begin{cases}
0, & m\text{ odd}\\
\frac{m}{2}\,\varphi_{\frac{m}{2}-1}(x), & m\text{ even}.
\end{cases}
\]
These relations follow directly from the definition \eqref{um}. One can write
them also as follows%
\[
u_{2n-1}(x,0)=\varphi_{n}(x),\qquad u_{2n}(x,0)=0,
\]
and
\[
\frac{\partial u_{2n-1}(x,0)}{\partial t}=0,\qquad\frac{\partial u_{2n}%
(x,0)}{\partial t}=n\varphi_{n-1}(x),\qquad\text{for }n=1,2,\ldots.
\]
For $u_{0}$ we have
\[
u_{0}(x,0)=f(x)=\varphi_{0}(x)\qquad\text{and}\qquad\frac{\partial u_{0}%
(x,0)}{\partial t}=0.
\]

\end{remark}

\section{Solution of the Cauchy problem\label{SectCauchyProblem}}

Consider the following initial value problem%
\begin{equation}
\square u-q(x)u=0,\qquad-b\leq x\leq b,\quad t\geq0 \label{Cauchy1}%
\end{equation}%
\begin{equation}
u(x,0)=g(x),\quad u_{t}(x,0)=h(x) \label{Cauchy2}%
\end{equation}
\begin{wrapfigure}[7]{r}{125\unitlength}
\begin{picture}(125,75)
\put(0,10){\vector(1,0){125}} \put(60,0){\vector(0,1){75}}
\put(117,1){$x$}
\put(52,67){$t$}
\multiput(15,10)(90,0){2}%
{\circle*{2}}
\put(8,0){$-b$} \put(105,0){$b$}
\put(60,55){\circle*{2}}
\put(63,54){$b$}
\thicklines
\put(15,10){\line(1,0){90}}
\put(15,10){\line(1,1){45}}
\put(105,10){\line(-1,1){45}}
\thinlines
\put(21,10){\line(1,2){6}}
\put(27,10){\line(1,2){12}}
\put(33,10){\line(1,2){18}}
\put(39,10){\line(1,2){22}}
\put(45,10){\line(1,2){20}}
\put(51,10){\line(1,2){18}}
\put(57,10){\line(1,2){16}}
\put(63,10){\line(1,2){14}}
\put(69,10){\line(1,2){12}}
\put(75,10){\line(1,2){10}}
\put(81,10){\line(1,2){8}}
\put(87,10){\line(1,2){6}}
\put(93,10){\line(1,2){4}}
\put(99,10){\line(1,2){2}}
\end{picture}
\end{wrapfigure}which for $q\in C[-b,b]$, $g\in C^{2}[-b,b]$, $h\in
C^{1}[-b,b]$ possesses a unique solution (see, e.g., \cite[Sect.
15.4]{Vladimirov}) in the triangle with the vertices $(\pm b,0)$ and $(0,b)$
(see illustration). For the convenience, later in this article we denote this
triangle by the symbol $\blacktriangle$. We assume that $q$ satisfies the
conditions of Theorem \ref{Th Transmutation of Powers K} and begin with the
additional assumption that the functions $g$ and $h$ admit uniformly
convergent series expansions in terms of the functions $\varphi_{k}$,
\begin{equation}
g(x)=\sum\limits_{k=0}^{\infty}\alpha_{k}\varphi_{k}(x)\qquad\text{and}\qquad
h(x)=\sum\limits_{k=0}^{\infty}\beta_{k}\varphi_{k}(x).
\label{g and h as series}%
\end{equation}
We look for a solution of the problem (\ref{Cauchy1}), (\ref{Cauchy2}) in the
form%
\begin{equation}
u(x,t)=\sum\limits_{n=0}^{\infty}a_{n}u_{n}(x,t).
\label{u in the form of a series}%
\end{equation}
Then we have (see Remark \ref{Rem u at 0})
\[
u(x,0)=a_{0}\varphi_{0}(x)+\sum\limits_{n=1}^{\infty}a_{2n-1}u_{2n-1}%
(x,0)=a_{0}\varphi_{0}(x)+\sum\limits_{n=1}^{\infty}a_{2n-1}\varphi_{n}(x)
\]
and
\[
u_{t}(x,0)=\sum\limits_{n=1}^{\infty}a_{2n}\frac{\partial u_{2n}%
(x,0)}{\partial t}=\sum\limits_{n=1}^{\infty}a_{2n}n\varphi_{n-1}%
(x)=\sum\limits_{k=0}^{\infty}a_{2(k+1)}(k+1)\varphi_{k}(x).
\]
Thus, if a solution of the problem (\ref{Cauchy1}), (\ref{Cauchy2}) in the
form (\ref{u in the form of a series}) exists, the expansion coefficients are
obtained directly from the coefficients in the expansions
(\ref{g and h as series}) as follows
\begin{equation}
a_{0}=\alpha_{0}\text{,\quad}a_{2n-1}=\alpha_{n}\text{,\quad}n=1,2,\ldots
\text{\qquad and \quad}a_{2(n+1)}=\frac{\beta_{n}}{n+1}\text{,\qquad
}n=0,1,2,\ldots\text{.} \label{coefficients a}%
\end{equation}

The following natural questions arise. Under which conditions given functions
$g$ and $h$ are representable in the form (\ref{g and h as series}) and
whether such series expansion is unique? Can one guarantee the uniform
convergence of the series (\ref{u in the form of a series}) and that of its
first and second derivatives in a domain of interest? In what follows we
address these questions and show that the described scheme also leads to a
powerful numerical technique for solving the initial value problems for
equation (\ref{Cauchy1}).

\begin{proposition}
\label{Prop Taylor series}A continuous complex-valued function $g$ defined on
$[-b,b]$ admits a series expansion of the form $g(x)=\sum_{k=0}^{\infty}%
\alpha_{k}\varphi_{k}(x)$ uniformly convergent on $[-b,b]$ if and only if
there exists a complex-valued function $\widetilde{g}$ defined on $[-b,b]$
such that $g=\mathbf{T}\widetilde{g}$, $\widetilde{g}(x)=\sum_{k=0}^{\infty
}\alpha_{k}x^{k}$ and the power series converges uniformly on $[-b,b]$. The
expansion coefficients are uniquely defined by the equalities
\begin{equation}
\alpha_{k}=\frac{d_{k}^{f}(g)(0)}{k!}=\frac{\widetilde{g}^{\left(  k\right)
}(0)}{k!}. \label{alphaK}%
\end{equation}

\end{proposition}

\begin{proof}
The proof of the representability of $g$ in the form of a uniformly convergent
series $\sum_{k=0}^{\infty}\alpha_{k}\varphi_{k}(x)$ follows from the uniform
boundedness of the Volterra integral operators $\mathbf{T}$ and $\mathbf{T}%
^{-1}$. The linearity of these integral operators together with the fact that
$\mathbf{T}\left[  x^{k}\right]  =\varphi_{k}$ (Theorem \ref{Th Transmute})
gives us the equality between the coefficients of the corresponding series
expansions of $g$ and $\widetilde{g}$. The equality $d_{k}^{f}%
(g)(0)=\widetilde{g}^{\left(  k\right)  }(0)$, $k=0,1,\ldots$ is a consequence
of Proposition \ref{GeneralDerivTransm} and of the observation that at the
origin $\mathbf{T}\left[  u\right]  (0)=\mathbf{T}_{D}\left[  u\right]
(0)=u(0)$ for any continuous function $u$.
\end{proof}

\begin{proposition}
Suppose that $g\in C[-b,b]$ and for any $k\in\mathbb{N}$ and $x\in(-b,b)$
there exists the generalized derivative $d_{k}^{f}(g)(x)$ such that for any
$[-a,a]\subset(-b,b)$ the inequality holds
\[
\left\vert d_{k}^{f}(g)\right\vert \leq C(a;k)\frac{k!}{b^{k}}%
\]
where the constants $C(a;k)$ do not depend on $x$ and the sequence $C(a;k)$ is
of a subexponential growth ($\overline{\lim}\sqrt[k]{C(a;k)}\leq1$). Then on
$(-b,b)$ $\ $the function $g$ admits a normally convergent generalized Taylor
series expansion
\begin{equation}
g(x)=\sum\limits_{k=0}^{\infty}\alpha_{k}\varphi_{k}(x) \label{genTaylor}%
\end{equation}
and $\alpha_{k}=d_{k}^{f}(g)(0)/k!$.
\end{proposition}

\begin{proof}
Under the conditions of the proposition consider $\widetilde{g}=\mathbf{T}%
^{-1}g$. From Proposition \ref{GeneralDerivTransm} we have that $\widetilde
{g}\in C^{\infty}(-b,b)$ and
\[
\left\vert \widetilde{g}^{\left(  k\right)  }\right\vert \leq M\,C(a;k)\frac
{k!}{b^{k}}%
\]
where $M=\max\left\{  \left\Vert \mathbf{T}^{-1}\right\Vert ,\,\left\Vert
\mathbf{T}_{D}^{-1}\right\Vert \right\}  $. Indeed, considering, e.g., an even
$k$ we obtain $\left\vert \widetilde{g}^{\left(  k\right)  }\right\vert
=\left\vert \mathbf{T}^{-1}d_{k}^{f}(g)\right\vert \leq\left\Vert
\mathbf{T}^{-1}\right\Vert \max\left\vert d_{k}^{f}(g)\right\vert
\leq\left\Vert \mathbf{T}^{-1}\right\Vert C(a;k)\frac{k!}{b^{k}}$ and
analogously for an odd $k$.

From this we obtain that $\widetilde{g}$ admits on $(-b,b)$ a normally
convergent Taylor series expansion of the form $\widetilde{g}(x)=\sum
_{k=0}^{\infty}\alpha_{k}x^{k}$ and due to Proposition
\ref{Prop Taylor series}, $g$ admits a normally convergent generalized Taylor
series expansion (\ref{genTaylor}).
\end{proof}

\begin{proposition}
Let the initial data $g$ and $h$ admit uniformly convergent series expansions
of the form \eqref{g and h as series} on $[-b,b]$. Then the unique (classical)
solution of the Cauchy problem \eqref{Cauchy1}, \eqref{Cauchy2} in
$\blacktriangle$ has the form \eqref{u in the form of a series} which is
uniformly convergent in $\blacktriangle$. The expansion coefficients are
defined by \eqref{coefficients a}.
\end{proposition}

\begin{proof}
As was previously shown if the series (\ref{u in the form of a series})
together with the series corresponding to the first and second partial
derivatives are normally convergent it satisfies equation (\ref{Cauchy1}) as
well as the conditions (\ref{Cauchy2}). Thus, it remains to prove the uniform
convergence of the involved series.

Using (\ref{umTpm}) we have $\left\vert u_{n}\right\vert \leq\left\Vert
\mathbf{T}\right\Vert \cdot\max\left\vert p_{n}\right\vert $, and from
(\ref{Re and Im}) we obtain
\[
\left\vert u_{n}\right\vert \leq b^{n}\left\Vert \mathbf{T}\right\Vert \text{
}%
\]
in the triangle $\blacktriangle$. Now, taking into account the uniform
convergence of the series (\ref{g and h as series}) on $[-b,b]$ we obtain the
uniform convergence of the series (\ref{u in the form of a series}) in
$\blacktriangle$. The series corresponding to the first and second partial
derivatives can be majorized in a similar way with the aid of Remarks
\ref{Rem Second derivative and Factorization} and \ref{Rem Derivatives phi}.
\end{proof}

As was mentioned above in \cite{KMoT} it was proved that any continuous and
piecewise continuously differentiable function on $[-b,b]$ can be approximated
arbitrarily closely by a finite linear combination of the functions
$\varphi_{k}$. The existence of a transmutation operator allows to show that
the condition of piecewise continuous differentiability is superfluous and provides a simple proof of the following proposition.

\begin{proposition}
\label{PropApproximabilityByPhi} Under the conditions of Theorem
\ref{Th Transmutation of Powers K} the system $\{\varphi_{k}\}_{k=0}^{\infty}$
is complete in $C[-b,b]$, i.e., any continuous function on $[-b,b]$ can be
approximated arbitrarily closely by a finite linear combination of the
functions $\varphi_{k}$.
\end{proposition}

\begin{proof}
The proof immediately follows from the existence of the transmutation
operator, Theorem \ref{Th Transmute} and the Weierstrass approximation theorem.
\end{proof}

Thus, even when it is not possible to guarantee the representability of the
initial data $g$ and $h$ in the form of uniformly convergent series
(\ref{g and h as series}), they can be approximated by corresponding $f$-
polynomials. The following statement gives us an estimate of the accuracy of
the solution $u$ of the problem (\ref{Cauchy1}), (\ref{Cauchy2}) approximated
by a solution $u_{N}$ corresponding to the approximated initial data.

\begin{proposition}
\label{Prop Approx Sol Estimate} Let $P_{n}(x)=\sum_{k=0}^{n}\alpha_{k}%
\varphi_{k}(x)$ and $Q_{n-1}(x)=\sum_{k=0}^{n-1}\beta_{k}\varphi_{k}(x)$ be
$f$-polynomials, approximating the functions $g$ and $h$ respectively on
$[-b,b]$ in such a way that $\max\left\vert g-P_{n}\right\vert <\varepsilon
_{1}$ and $\max\left\vert h-Q_{n-1}\right\vert <\varepsilon_{2}$. Let
$u_{N}(x,t)=\sum_{k=0}^{N}a_{k}u_{k}(x,t)$, $N=2n$ where $a_{0}=\alpha_{0}$,
$a_{2m-1}=\alpha_{m}$, for $m=1,2,\ldots,n$ and $a_{2(m+1)}=\frac{\beta_{m}%
}{m+1}$, for $m=0,1,\ldots,n-1$. Then
\begin{equation}
\max_{(x,t)\in\blacktriangle}\left\vert u-u_{N}\right\vert <\left\Vert
\mathbf{T}\right\Vert \left\Vert \mathbf{T}^{-1}\right\Vert \left(
\varepsilon_{1}+\varepsilon_{2}b\right)  . \label{estimate}%
\end{equation}

\end{proposition}

\begin{proof}
Notice that $u_{N}$ is a solution of the Cauchy problem for equation
(\ref{Cauchy1}) with the initial conditions $u_{N}(x,0)=P_{n}(x)$,
$u_{N,t}(x,0)=Q_{n-1}(x)$, $x\in\lbrack-b,b]$. Consider the functions
$\widetilde{u}=\mathbf{T}^{-1}[u]$ and $\widetilde{u}_{N}=\mathbf{T}%
^{-1}[u_{N}]$. They solve the wave equation (\ref{wave equation}) and satisfy
the initial conditions $\widetilde{u}(x,0)=\widetilde{g}(x)$, $\widetilde
{u}_{t}(x,0)=\widetilde{h}(x)$ and $\widetilde{u}_{N}(x,0)=\widetilde{P}%
_{n}(x)$, $\widetilde{u}_{N,t}(x,0)=\widetilde{Q}_{n-1}(x)$, $x\in
\lbrack-b,b]$ where the tilde indicates the image of a corresponding function
under the action of $\mathbf{T}^{-1}$. We have $\max\left\vert \widetilde
{g}-\widetilde{P}_{n}\right\vert <\left\Vert \mathbf{T}^{-1}\right\Vert
\varepsilon_{1}$ and $\max\left\vert \widetilde{h}-\widetilde{Q}%
_{n-1}\right\vert <\left\Vert \mathbf{T}^{-1}\right\Vert \varepsilon_{2}$.
From the d'Alembert formula by analogy with the standard proof of the
stability of the Cauchy problem for the wave equation (see, e.g., \cite[Sect.
4.3]{Pinchover}) we obtain $\max_{(x,t)\in\blacktriangle}\left\vert
\widetilde{u}-\widetilde{u}_{N}\right\vert <\left\Vert \mathbf{T}%
^{-1}\right\Vert \left(  \varepsilon_{1}+\varepsilon_{2}b\right)  $. Finally,
(\ref{estimate}) is obtained from the following $\max_{(x,t)\in\blacktriangle
}\left\vert u-u_{N}\right\vert =\max_{(x,t)\in\blacktriangle}\left\vert
\mathbf{T}\left(  \widetilde{u}-\widetilde{u}_{N}\right)  \right\vert
\leq\left\Vert \mathbf{T}\right\Vert \max_{(x,t)\in\blacktriangle}\left\vert
\widetilde{u}-\widetilde{u}_{N}\right\vert $.
\end{proof}

\section{Approximation by the functions $\{\varphi_{k}\}_{k=0}^{\infty}%
$\label{SectApprox}}

It follows from Proposition \ref{Prop Approx Sol Estimate} that approximations
of continuous functions by $f$- polynomials play a significant role in
constructing approximate solutions of the Cauchy problem
\eqref{Cauchy1}--\eqref{Cauchy2}. The transmutation operator and the relation
between the functions $\varphi_{k}$ and the powers $x^{k}$ made it possible to
prove Proposition \ref{PropApproximabilityByPhi} showing that any continuous
function may be approximated arbitrarily closely by finite linear combinations
of the functions $\varphi_{k}$. In this section we use the transmutation
operators to extend some well-known results of approximation theory (see,
e.g.
\cite{DeVoreLorentz}, \cite{Dzyadyk}, \cite{Timan}) onto approximations by the
functions $\varphi_{k}$ and discuss different ways to construct such
approximations for a given function.

Denote by $\Phi_{n},\ n=0,1,\ldots$ the linear vector space spanned by the
functions $\varphi_{0},\ldots,\varphi_{n}$. It follows from Theorem
\ref{Th Transmute} that the functions $\varphi_{0},\ldots,\varphi_{n}$ are
linearly independent, therefore the space $\Phi_{n}$ is $(n+1)$-dimensional
and the embedding $\Phi_{n}\subset\Phi_{n+1}$ holds for any $n$.

Define by
\[
\mathcal{E}_{n}^{f}(g)=\min_{h_{n}\in\Phi_{n}}\Vert g-h_{n}\Vert
\]
the best approximation of a continuous function $g$ by $f$-polynomials of
degree $n$, i.e., by finite linear combinations $\sum_{k=0}^{n}\alpha
_{k}\varphi_{k}$ (Definition \ref{Defn Pn}). Here $\Vert\cdot\Vert$ denotes
the usual uniform norm on $[-b,b]$. Due to the embedding $\Phi_{n}\subset
\Phi_{n+1}$ the quantity $\mathcal{E}_{n}^{f}(g)$ is monotone decreasing as
$n\rightarrow\infty$. Proposition \ref{PropApproximabilityByPhi} states that
$\mathcal{E}_{n}^{f}(g)\rightarrow0,\ n\rightarrow\infty$ for any function
$g\in C[-b,b]$. It is known in the approximation theory that for some
functions $g$ the convergence rate of $\mathcal{E}_{n}^{f}(g)$ to zero may
result to be arbitrarily slow. Nevertheless additional smoothness properties
of the function $g$ allow one to obtain more precise results on this
convergence rate.

\begin{theorem}
[Direct approximation theorem]\label{Th DirectApprox} Suppose the function $g$
possesses on the segment $[-b,b]$ continuous $f$-derivatives of all orders up
to the order $k$. Then for the best approximation by $f$-polynomials the
following estimates hold for any $n\ge k$
\[
\mathcal{E}_{n}^{f}(g)\le\frac{\Big(\dfrac{\pi b}{2}\Big)^{k} \|\mathbf{T}%
\|\max\{\|\mathbf{T}^{-1}\|, \|\mathbf{T}_{D}^{-1}\|\}}{(n+1)n\cdot\ldots
\cdot(n-k+2)}\|d_{k}^{f}g\|
\]
and
\[
\mathcal{E}_{n}^{f}(g)=\frac{o(1)}{n^{k}},\quad\text{as }n\to\infty.
\]

\end{theorem}

\begin{proof}
Consider the function $\widetilde{g}=\mathbf{T}^{-1}g$. As follows from
Proposition \ref{GeneralDerivTransm}, $\widetilde{g}\in C^{k}[-b,b]$. A
variant of Jackson's theorem \cite[Chap.4, Sec.6]{Cheney} states that
\[
E_{n}(\widetilde{g})\leq\frac{1}{(n+1)n\cdot\ldots\cdot(n-k+2)}\Big(\dfrac{\pi
b}{2}\Big)^{k}\Vert\widetilde{g}^{(k)}\Vert
\]
where $E_{n}$ denotes the best approximation by algebraic polynomials of
degree $\leq n$. Due to Proposition \ref{GeneralDerivTransm} we have
$\Vert\widetilde{g}^{(k)}\Vert\leq\max\{\Vert\mathbf{T}^{-1}\Vert
,\Vert\mathbf{T}_{D}^{-1}\Vert\}\cdot\Vert d_{k}^{f}g\Vert$. Now the first
statement of the theorem follows from Theorem \ref{Th Transmute}.

The second statement easily follows from another variant of Jackson's theorem
\cite[VI.2]{Dzyadyk}, \cite[5.2.1]{Timan}: if the function $h\in C^{k}[-b,b]$,
then for any $n\ge k$
\[
E_{n}(h)\le A\big(\rho_{n}(x)\big)^{k}\omega\big(\rho_{n}(x)\big),
\]
where $\rho_{n}(x)=\frac{\sqrt{(b-x)(x+b)}}{n}+\frac1{n^{2}}$, $\omega
(t):=\omega(h^{(k)}; t)$ is the modulus of continuity of the derivative
$h^{(k)}$, satisfying $\omega(t)\to0,\ t\to0$, and the constant $A$ does not
depend on $h$ and $n$.
\end{proof}

\begin{remark}
A similar result holds under a weaker condition on the smoothness of the
function $g$, namely, suppose that $g$ possesses on the segment $[-b,b]$
continuous $f$-derivatives of all orders up to the order $k-1$ and the
$f$-derivative of the order $k-1$ is Lipschitz continuous on $[-b,b]$, i.e.,
$|d_{k-1}^{f}g(x)-d_{k-1}^{f}g(y)|\leq M|x-y|$ for some constant $M$ and for
every $x,y\in\lbrack-b,b]$. Then there exists a constant $C>0$ such that
\[
\mathcal{E}_{n}^{f}(g)\leq\frac{C}{n^{k}}\quad\text{for any }n\geq k.
\]
The proof may be done similarly to the proof of Theorem \ref{Th DirectApprox}
with the use of Jackson's theorem \cite[Chap.4, Sec.6]{Cheney} or
\cite[5.2.4]{Timan} and the fact that if a function $g$ is Lipschitz
continuous, then the function $\widetilde{g}=\mathbf{T}^{-1}g$ is Lipschitz
continuous as well.
\end{remark}

The classical reasoning in the proof of an inverse theorem for the function
$\widetilde{g}=\mathbf{T}^{-1}g$ with the application of Markov's inequality
and of an inequality for the derivative of the polynomial (see, e.g.
\cite[4.8.7 and 6.2]{Timan}, \cite[VII.2]{Dzyadyk}) allows us to prove a
partial reverse statement of Theorem \ref{Th DirectApprox}. We show that the
obtained convergence rate of the best approximations is close to optimal.

\begin{theorem}
[Inverse approximation theorem]\label{Th InverseApprox} Suppose that the best
approximations by $f$-polynomials of some function $g$ satisfy for some
integer number $r$ and positive constants $M$ and $\varepsilon$ the
inequality
\begin{equation}
\label{EnfEstim}\mathcal{E}_{n}^{f}(g)\le\frac{M}{n^{r+\varepsilon}}%
\qquad\forall n\in\mathbb{N}.
\end{equation}
Then the function $g$ possesses $f$-derivatives of order $r$ in $(-b,b)$ and
$f$-derivatives of order at least $[r/2]$ at the endpoints, where $[\cdot]$
denotes the integer part of a number.
\end{theorem}

\begin{proof}
Consider the function $\widetilde{g}=\mathbf{T}^{-1}g$. As it follows from
\eqref{EnfEstim} and Theorem \ref{Th Transmute}, there exists a sequence of
polynomials $P_{n}$ such that
\begin{equation}
\Vert\widetilde{g}-P_{n}\Vert\leq\frac{\widetilde{M}}{n^{r+\varepsilon}}%
\qquad\forall n\in\mathbb{N},\label{g-pn}%
\end{equation}
where $\widetilde{M}=M\Vert\mathbf{T}^{-1}\Vert$. Consider the series
\begin{equation}
P_{1}(x)+\sum_{k=0}^{\infty}\big(P_{2^{k+1}}(x)-P_{2^{k}}%
(x)\big).\label{Series Pn}%
\end{equation}
It is uniformly convergent due to the estimate
\begin{equation}
\big|P_{2^{k+1}}(x)-P_{2^{k}}(x)\big|\leq\big|\widetilde{g}-P_{2^{k}%
}\big|+\big|P_{2^{k+1}}-\widetilde{g}\big|\leq\frac{\widetilde{M}%
}{2^{k(r+\varepsilon)}}+\frac{\widetilde{M}}{2^{(k+1)(r+\varepsilon)}}%
\leq\frac{2\widetilde{M}}{2^{k(r+\varepsilon)}},\label{pk1-pk}%
\end{equation}
and as it is easy to see, the sum of the series is equal to $\widetilde{g}$.
To finish the proof, we use two well-known inequalities for the derivative of
the polynomial of order $n$ defined on $[-b,b]$. First,
\[
|P_{n}^{\prime}(x)|\leq\frac{n}{\sqrt{b^{2}-x^{2}}}\Vert P_{n}(x)\Vert
\]
and Markov's inequality
\[
|P_{n}^{\prime}(x)|\leq\frac{n^{2}}{b}\Vert P_{n}(x)\Vert.
\]
From the first inequality and estimate \eqref{pk1-pk} we obtain for any
segment $[-d,d]\subset(-b,b)$
\[
\big|P_{2^{k+1}}^{(r)}(x)-P_{2^{k}}^{(r)}(x)\big|\leq\frac{2\widetilde{M}%
C_{d}\cdot2^{(k+1)r}}{2^{k(r+\varepsilon)}}=\frac{2^{r+1}\widetilde{M}C_{d}%
}{2^{k\varepsilon}},
\]
where the constant $C_{d}$ depends only on $r$ and the segment $[-d,d]$. The
obtained estimate leads to the uniform convergence of the series of $r$-th
derivatives of \eqref{Series Pn} and  hence to the conclusion that
$\widetilde{g}\in C^{r}(-b,b)$. Similarly, the second inequality leads to the
conclusion that $\widetilde{g}\in C^{[r/2]}[-b,b]$. Application of Proposition
\ref{GeneralDerivTransm} finishes the proof.
\end{proof}

Contrary to the $L_{2}$-norm, the problem of explicit finding of a polynomial
of the best uniform approximation can be solved in some special cases only.
But from a practical point of view the exact solution is not that necessary,
it is enough to know a polynomial which is sufficiently close to the best one.
Techniques such as least squares approximation or the Lagrange interpolation
(with specially chosen nodes) work well though in general far from the best,
see \cite{Rivlin}. Below we briefly describe the iterative algorithm of E.
Remez for constructing polynomials arbitrarily close to the best one. Even the
zero step of the algorithm, the so-called Tchebyshev interpolation, usually
gives better results then the Lagrange interpolation. For a detailed
description of the algorithm with implementation details and all the required
proofs we refer to \cite{Remez}, \cite{Meinardus}, \cite{Cheney}.

First we remind some definitions and statements related to Tchebyshev uniform
approximations. See
\cite{DeVoreLorentz}, \cite{Timan} for details.

A linear subspace $V$ of $C[-b,b]$ of (finite) dimension $n+1$ is said to
fulfill the \textbf{Haar condition} if it possesses the property that every
function in $V$ which is not identically zero vanishes at no more than $n$
points of $[-b,b]$. An equivalent condition is that the interpolation problem
is uniquely solvable, i.e., for every set of $n+1$ points $x_{k}%
\ (k=0,1,\ldots,n)$ in $[-b,b]$ and every prescribed vector $(y_{0}%
,y_{1},\ldots,y_{n})$ there exists a unique function $h\in V$ such that
\[
h(x_{k})=y_{k},\quad k=0,1,\ldots,n.
\]
If $V$ is spanned by the functions $h_{0},h_{1},\ldots,h_{n}$, another
equivalent condition is that every determinant
\[%
\begin{vmatrix}
h_{0}(x_{0}) & h_{1}(x_{0}) & \ldots & h_{n}(x_{0})\\
h_{0}(x_{1}) & h_{1}(x_{1}) & \ldots & h_{n}(x_{1})\\
\vdots & \vdots & \ddots & \vdots\\
h_{0}(x_{n}) & h_{1}(x_{n}) & \ldots & h_{n}(x_{n})
\end{vmatrix}
\neq0
\]
for any distinct points $x_{0},x_{1},\ldots,x_{n}$ from $[-b,b]$. The Haar
condition is necessary and sufficient for the unique solvability of the
approximation problem.

A system of linearly independent functions $h_{0},h_{1},\ldots,h_{n}$ is
called a \textbf{Tchebyshev system} if the linear subspace spanned by these
functions satisfies the Haar condition.

\begin{proposition}
Let $f$ be a real-valued non-vanishing continuous function on $[-b,b]$. Then
the system of functions $\{\varphi_{k}\}_{k=0}^{\infty}$ constructed by
\eqref{phik} is a Markov system, i.e., for any $n\in\mathbb{N}_{0}$ the first
$n+1$ functions form a Tchebyshev system and the subspace $\Phi_{n}$ spanned
by these functions satisfies the Haar condition.
\end{proposition}

\begin{proof}
The proof by induction is straightforward using the fact that for the
$f$-derivative the Rolle theorem holds. Also the result may be deduced from
\cite[\S 3.11]{DeVoreLorentz} if we observe that for the real-valued
non-vanishing function $f$ the system $\{\varphi_{k}\}_{k=0}^{n}$ is a scaled
P\'{o}lya system.
\end{proof}

\begin{remark}
\label{HaarFails} As the following example shows, for $f$ being a
complex-valued function the Haar condition may fail for the subspaces
$\Phi_{n}$. Consider $f(x)=e^{ix}$. Then the first two functions $\varphi_{k}$
are $\varphi_{0}=e^{ix}$, $\varphi_{1}=\sin x$, and for large segments the
function $\varphi_{1}$ may have arbitrarily many zeroes.
\end{remark}

Assume that the function $f$ and hence all functions $\varphi_{k}$ are
real-valued (we briefly discuss the complex-valued case at the end of this section).

The Remez algorithm is based on the Tchebyshev theorem with a generalization
by de la Vall\'{e}e Poussin which gives a characterization of the polynomial
of the best approximation \cite[2.7.3]{Timan}.

\begin{theorem}
[Tchebyshev's alternance theorem]\label{Th Tchebyshev} If $P_{n} = \sum
_{k=0}^{n} c_{k}\varphi_{k}$ is a polynomial with respect to some Tchebyshev
system $\{\varphi_{k}\}_{k=0}^{n}$, $g$ is a continuous function and $Q$ is an
arbitrary closed subset of the segment $[a,b]$, then $P_{n}$ is the best
approximation to $g$ on $Q$ if and only if the difference $g(x)-P_{n}(x)$
attains a maximum of its modulus on $Q$, with alternative signs, at least at
$n+2$ distinct points of the given set.
\end{theorem}

Such set of $n+2$ points is often called an alternant of the function $g$. An
important consequence of this theorem is that for any continuous function $g$
there exist exactly $n+2$ distinct points $\xi_{0},\ldots,\xi_{n+1}$ from
$[-b,b]$ such that the best approximation of $g$ on the whole segment $[-b,b]$
coincides with the best approximation on this so-called characteristic set of
$n+2$ points. The idea of the Remez algorithm is to construct iteratively
subsets of $[-b,b]$ each of them consisting of $n+2$ points in such a way that
on every step the value of the best approximation on the $n+2$ points subset
be increasing.

In the case when the set $Q$ consists of exactly $n+2$ distinct points
$x_{0}<x_{1}<\ldots<x_{n+1}$, the problem of determination of the best
approximation polynomial $P_{n}$ of $g$ on $Q$ is exactly solvable and reduces
to the solution of the system of $n+2$ linear equations
\begin{equation}
\sum_{k=0}^{n}c_{k}\varphi_{k}(x_{j})+(-1)^{j}E(g)=g(x_{j}),\qquad
j=0,1,\ldots,n+1\label{TchebInterp}%
\end{equation}
for the coefficients $c_{k},\ k=0,\ldots,n$ and the value of the best
approximation $E(g)=\mathcal{E}_{n}(g)$ on the set $Q$. The solution of the
problem \eqref{TchebInterp} for given points $x_{0}<x_{1}<\ldots<x_{n+1}$ and
values of the function $g(x_{j})$ in these points is also called
\textbf{Tchebyshev interpolation}. Note that unlike the Lagrange
interpolation, the resulted polynomial does not pass exactly through the given
values of the function but the deviations of the polynomial from the given
values are equal by absolute value at all points and differ only in sign.

Let us describe the iterative algorithm of E. Remez. We are looking for an
$f$-polynomial close to the one giving the best approximation $\mathcal{E}%
_{n}^{f}(g)$ of a given function $g$ by polynomials from $\Phi_{n}$.

We begin with a set $M_{0}$ consisting of $n+2$ distinct points $-b\leq
x_{0}^{(0)}<x_{1}^{(0)}<\ldots<x_{n+1}^{(0)}\leq b$. Corresponding to these
points, using \eqref{TchebInterp} we construct an $f$-polynomial of Tchebyshev
interpolation $g_{0}=\sum_{k=0}^{n}c_{k}\varphi_{k}\in\Phi_{n}$. The function
$g_{0}(x)$ is the best approximation of $g(x)$ on the set $M_{0}$. Denote the
value $E(g)$ obtained from \eqref{TchebInterp} by $E_{0}(g)$, and let
$D_{0}:=\Vert g-g_{0}\Vert$. It follows from Theorem \ref{Th Tchebyshev} and
from the observation that the best approximation on $n+2$ points subset is not
worse than the best approximation on the whole segment $[-b,b]$ that
\[
|E_{0}(g)|\leq\mathcal{E}_{n}^{f}(g)\leq D_{0}=\Vert g-g_{0}\Vert.
\]
Now either $\Vert g_{0}-g\Vert=|E_{0}(g)|$ and we are done, or $\Vert
g-g_{0}\Vert>|E_{0}(g)|$. The idea of E. Remez is to construct a new set
$M_{1}$ which again consists of $n+2$ points, but for which the corresponding
linear functional $E_{1}(g)$ has a larger magnitude than $|E_{0}(g)|$.

There are two possibilities to define the set $M_{1}$. The first is the
so-called single exchange method. Exactly one of the points of $M_{0}$ is
replaced by a new point $\xi$ satisfying $|g(\xi)-g_{0}(\xi)|=\Vert
g-g_{0}\Vert$. The point to be removed is chosen in such a way that the
difference $g-g_{0}$ alternates in sign at the points of the new sequence, it
is not hard to derive an exact table of rules. Renumeration of the points
according to their magnitudes produces the set $M_{1}$.

The second possibility is the general method of E. Remez. It involves
simultaneous exchanges. The function $h_{0}:=g-g_{0}$ possesses at least $n+1$
zeroes $z_{k}^{(0)},\ k=1,\ldots,n+1$ in the interval $(-b,b)$ and
\[
x_{k}^{(0)}<z_{k+1}^{(0)}<x_{k+1}^{(0)},\qquad k=0,1,\ldots,n.
\]
Set $z_{0}^{(0)}=-b$, $z_{n+2}^{(0)}=b$. Now in each interval $J_{k}%
=\big[z_{k}^{(0)},z_{k+1}^{(0)}\big],\ k=0,\ldots,n+1$ we determine a point
$x_{k}^{(1)}$ such that
\[
h_{0}\big(x_{k}^{(1)}\big)\geq h_{0}(x)\qquad\text{for all }x\in
J_{k}\ \text{if}\ \operatorname{sgn}h_{0}\big(x_{k}^{(0)}\big)=1,
\]
and
\[
h_{0}\big(x_{k}^{(1)}\big)\leq h_{0}(x)\qquad\text{for all }x\in
J_{k}\ \text{if}\ \operatorname{sgn}h_{0}\big(x_{k}^{(0)}\big)=-1,
\]
that is, we are looking for a maximum if the difference between $g$ and the
previous approximation is positive, and for a minimum, if the difference is
negative. Note that corresponding maxima and minima always exist. Here we
assumed that $E_{0}(g)\neq0$. If $E_{0}(g)=0$ the points $x_{k+1}^{(1)}$ are
to be chosen as a sequence of points at which $h_{0}(x)$ has alternatively a
maximum and a minimum.

The iteration is repeated until the quantity $\dfrac{D_{k}-E_{k}(g)}{D_{k}}$,
characterizing the closeness of the found $f$-polynomial to the best one, is
not sufficiently small.

Under the condition that in each of the sets $M_{m+1}$ there is a point $\xi$
such that $|h_{m}(\xi)|=\Vert h_{m}\Vert$ both the single and the general
exchange algorithms converge to the best approximation. The convergence speed
is at least linear, i.e., there exists a constant $q<1$ such that
\[
\mathcal{E}_{n}^{f}(g)-|E_{m+1}(g)|\leq q\big(\mathcal{E}_{n}^{f}%
(g)-|E_{m}(g)|\big)
\]
(see \cite{Meinardus} for details). Under some additional assumptions on the
smoothness of the function $g$ and the functions $\{\varphi_{k}\}_{k=0}^{n}$
and the number and type of extremal points of the difference $h=g-\widetilde
{g}$ in $[-b,b]$, where $\widetilde{g}$ is the $f$-polynomial of the best
approximation, the convergence rate is quadratic \cite[Thm. 84]{Meinardus}.
I.e.,  for practical purposes only few iterations are required.

As with any iterative algorithm, an important question is to choose properly a
good initial set $M_{0}$. One of the possibilities is to consider the function
$\widehat{g}$ of the best least-square approximation to $g$ with respect to
the functions $\varphi_{0},\ldots,\varphi_{n}$. It is known \cite[p.
129]{Meinardus} that if the difference $g-\widehat{g}$ does not vanish
identically on $[-b,b]$ then it possesses at least $n+1$ zeroes on $[-b,b]$,
hence it has at least $n+2$ alternating points of maxima and minima. The
coordinates of these extremal points may be considered as the starting set
$M_{0}$.

Another possibility (see \cite[4.1 and 7.2]{Meinardus}) is recommended if it
is necessary to construct approximations of several functions with respect to
the same functions $\varphi_{0},\ldots,\varphi_{n}$. We consider the problem
of finding the best approximation $\widetilde{\varphi}$ of the function
$\varphi_{n+1}$ by the functions $\varphi_{0},\ldots,\varphi_{n}$. The
function $s:=\varphi_{n+1}-\widetilde{\varphi}$ is not identically zero and
possesses exactly $n+2$ extremal points. These extremal points form a good
initial set for the Remez algorithm. In the case when the functions
$\varphi_{k}$ coincide with the powers $x^{k}$, the function $s$ coincides (up
to a constant factor) with the Tchebyshev polynomial $T_{n+1}(x/b)$ and the
extremal points are given by $x_{k}=-b\cos\frac{k\pi}{n+1},\ k=0,\ldots,n+1$.

It is worth mentioning that the approximation problem may be discretized and
interpreted as a linear programming problem and solved by available software.
We take a finite subset $X\subset\lbrack-b,b]$ consisting of points
$x_{1},\ldots,x_{N}$, where $N\geq n+2$. The condition
\[
\max_{x\in X}\bigg|g(x)-\sum_{k=0}^{n}c_{k}\varphi_{k}(x)\bigg|=E
\]
can be written as
\[
-E\leq g(x_{j})-\sum_{k=0}^{n}c_{k}\varphi_{k}(x_{j})\leq E,\qquad
j=1,\ldots,N.
\]
Our problem is to \emph{minimize} the linear function $1\cdot E+0\cdot
c_{0}+\ldots+0\cdot c_{n}$ subject to $2N$ linear constraints
\begin{align*}
E+\sum_{k=0}^{n}c_{k}\varphi(x_{j}) &  \geq g(x_{j}),\qquad j=1,\ldots,N\\
E-\sum_{k=0}^{n}c_{k}\varphi(x_{j}) &  \geq-g(x_{j}),\qquad j=1,\ldots,N.
\end{align*}
The obtained problem can be solved by a variety of methods available for
solving linear programming problems, see \cite{Rice}, \cite{Rivlin} for details.


At the end of this section we return to the case of the complex-valued
function $f$. As was mentioned in Remark \ref{HaarFails} the Haar condition
may fail for the subspace $\Phi_{n}$. Even if the Haar condition holds, there
is no immediate generalization of the Remez algorithm for the complex-valued
case. The reason is that the Remez algorithm is based on the existence of a
characteristic set of a function consisting of exactly $n+2$ points. We remind
that a subset $X\subset\lbrack-b,b]$ is called characteristic for the function
$g$ if the best approximation of $g$ on the whole segment $[-b,b]$ coincides
with the best approximation on the subset $X$, but does not coincide on any
proper subset of $X$. Contrary to the real-valued case in the complex-valued
case a characteristic set may contain any number of points between $n+2$ and
$2n+3$, see \cite{SmirLeb}. There is no simple way to determine the number of
characteristic points for a given function. What is more, the given function
may have several characteristic sets containing different numbers of points.

In the existing algorithms the discretized problem is considered and solved
directly as a nonlinear optimization problem, e.g., a convex programming
problem \cite{BDM1978}, \cite{Watson1990}, or the problem is transformed into
a semi-infinite programming problem with the use of the fact that
$|h|=\max_{\phi\in\lbrack0,2\pi)}\operatorname{Re}\big(e^{i\phi}\cdot h\big)$.
The dual problem is considered and discretized for the second time with
respect to the angle $\phi$ and solved by the simplex method \cite{BDM1978},
\cite{GR1981} or by a Remez-like algorithm \cite{Kovtunets1987},
\cite{Kovtunets1988}, \cite{Tang1988}, \cite{FM1992}. If the obtained
approximation is not sufficiently close to the best one, the optimality
criterium of the best approximation \cite{Rivlin}, \cite{SmirLeb} is
reformulated as a system of nonlinear equations and the Newton iterations are
used to improve the accuracy, see \cite{GR1981}, \cite{Tang1988},
\cite{Watson1990}, \cite{FM1992} for details.

\section{Numerical examples}

In this section we present several numerical examples illustrating the
application of the described results on generalized wave polynomials and
approximation by functions $\{\varphi_{k}\}_{k=0}^{\infty}$ to numerical
solution of the Cauchy problem (\ref{Cauchy1}), (\ref{Cauchy2}). On the first
step the initial data $g$ and $h$ are approximated by $f$-polynomials and then
the approximate solution of the problem (\ref{Cauchy1}), (\ref{Cauchy2}), the
function $u_{N}$ from Proposition \ref{Prop Approx Sol Estimate}, is calculated
on a mesh of points in the triangle from the figure in Section
\ref{SectCauchyProblem} and compared to a corresponding exact solution.
All calculations were performed using Matlab in the machine precision. For the
construction of the system of the functions $\varphi_{k}$ the following
strategy was implemented using two Matlab routines from the Spline Toolbox: on
each step the integrand is approximated by a spline using the command
\texttt{spapi} and then it is integrated using \texttt{fnint}. This leads to a
good accuracy, and the computation of the first 180--200 or even more functions
$\varphi_{k}$ proved to be a completely feasible task. In all the reported
examples the number of subintervals in which the considered segment is divided
when the integrand is approximated by a spline was 3000 and the splines were
of the forth order. In the presented numerical results we specify the
parameter $n$ which is the number of the calculated functions $\varphi_{k}$.

\begin{example}
\label{Ex1}Consider the Cauchy problem
\begin{equation}
\square u-c^{2}u=0,\qquad-b\leq x\leq b,\quad t\geq0, \label{Cauchy1Ex1}%
\end{equation}%
\begin{equation}
u(x,0)=g(x)=\cosh\sqrt{c^{2}-\lambda_{1}^{2}}x,\quad u_{t}(x,0)=h(x)=1,\qquad
c,\lambda_{1}\in\mathbb{C}. \label{Cauchy2Ex1}%
\end{equation}
The exact solution of this problem has the form
\[
u(x,t)=\frac{1}{c}\sin ct+\cos\lambda_{1}t\cosh\sqrt{c^{2}-\lambda_{1}^{2}}x.
\]
The corresponding second-order ordinary differential equation (\ref{SLhom}),
$f^{\prime\prime}-c^{2}f=0$ admits a nonvanishing solution $f(x)=e^{cx}$,
$f(0)=1$. Based on this solution we construct $n$ functions $\varphi_{k}$
defined by (\ref{phik}) and (\ref{X1})--(\ref{X3}) with $x_{0}=0$. The initial
data for this example were chosen such that both $g$ and $h$ admit uniformly
convergent on $[-b,b]$ generalized Taylor series (see Subsection
\ref{Subsect Gen Der and Gen Taylor}) whose expansion coefficients are known
explicitly. Indeed, observe that $g$ and $h$ are solutions of the equation
$v^{\prime\prime}-c^{2}v=\lambda v$ with different values of the parameter
$\lambda$. In the case of $g$: $\lambda=-\lambda_{1}^{2}$ and in the case of
$h$: $\lambda=-c^{2}$. Since $g(0)=1$ and $g^{\prime}(0)=0$ due to Theorem
\ref{ThGenSolSturmLiouville copy(1)} the function $g$ can be represented as
follows
\[
g=u_{1}-cu_{2}=\sum_{k=0}^{\infty}\frac{(-\lambda_{1}^{2})^{k}}{(2k)!}%
\varphi_{2k}-c\sum_{k=0}^{\infty}\frac{(-\lambda_{1}^{2})^{k}}{(2k+1)!}%
\varphi_{2k+1}%
\]
where $u_{1}$ and $u_{2}$ are defined by (\ref{u1u2}) and the series are
uniformly convergent on $[-b,b]$ \ for any finite $b$. Thus, the coefficients
$\alpha_{k}$ from (\ref{g and h as series}) have the form%
\begin{equation}
\alpha_{2k}=\left(  -1\right)  ^{k}\frac{\lambda_{1}^{2k}}{(2k)!},\quad
\alpha_{2k+1}=\left(  -1\right)  ^{k+1}c\frac{\lambda_{1}^{2k}}{(2k+1)!},\quad
k=0,1,.... \label{alphaKEx1}%
\end{equation}
Analogously we obtain
\[
h=u_{1}-cu_{2}=\sum_{k=0}^{\infty}\frac{(-c^{2})^{k}}{(2k)!}\varphi_{2k}%
-c\sum_{k=0}^{\infty}\frac{(-c^{2})^{k}}{(2k+1)!}\varphi_{2k+1},
\]
and the coefficients $\beta_{k}$ from (\ref{g and h as series}) have the form
\begin{equation}
\beta_{2k}=(-1)^{k}\frac{c^{2k}}{(2k)!},\quad\beta_{2k+1}=(-1)^{k+1}%
\frac{c^{2k+1}}{(2k+1)!},\quad k=0,1,2.... \label{betaKEx1}%
\end{equation}

As an example let us take $b=2$, $c=3$ and $\lambda_{1}=1$. Consider the
$f$-polynomials $P_{n}$ and $Q_{n-1}$ from Proposition
\ref{Prop Approx Sol Estimate} with $n=20$ obtained by truncating the
generalized Taylor series (\ref{g and h as series}). Figure \ref{Fig1} depicts
the distribution of the absolute error of such approximation of the functions
$g$ and $h$. One can
observe that the apparently simplier function $h\equiv1$ is approximated much
worse ($10^{-3}$ against $10^{-8}$) by the truncated generalized Taylor polynomial.

\begin{figure}[ptb]
\begin{center}
\noindent \includegraphics[
height=2.25in,
width=3in
]%
{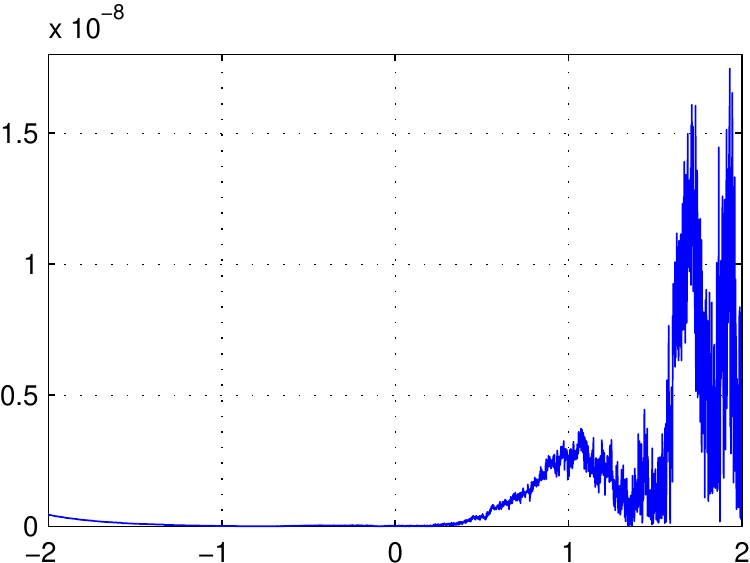}\quad
\includegraphics[
height=2.25in,
width=3in
]%
{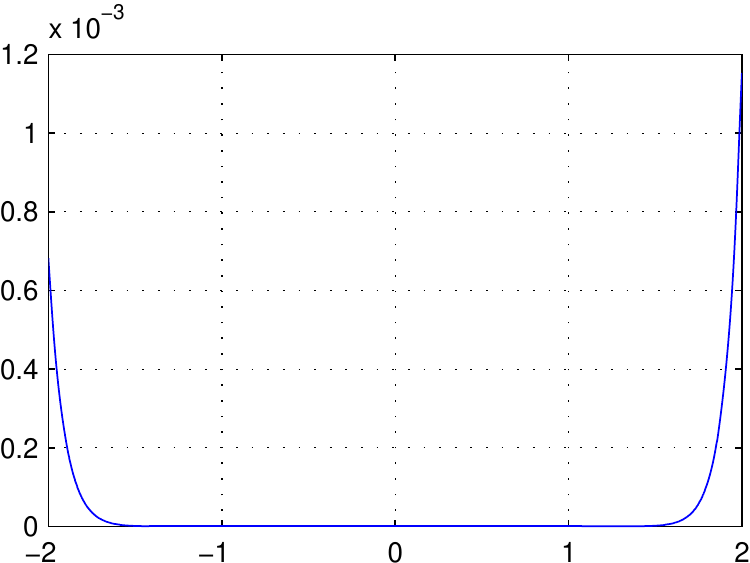}%
\caption{Graphs of $\left\vert g-P_{20}\right\vert $ (on the left) and $\left\vert h-Q_{19}\right\vert $ (on the right) from Example \ref{Ex1}
with the coefficients $\alpha_{k}$ and $\beta_{k}$ obtained by (\ref{alphaK}) which in this
case reduces to (\ref{alphaKEx1}) and (\ref{betaKEx1}).}%
\label{Fig1}%
\end{center}
\end{figure}

\begin{figure}[ptb]
\begin{center}
\noindent
\includegraphics[
height=2.25in,
width=3in]%
{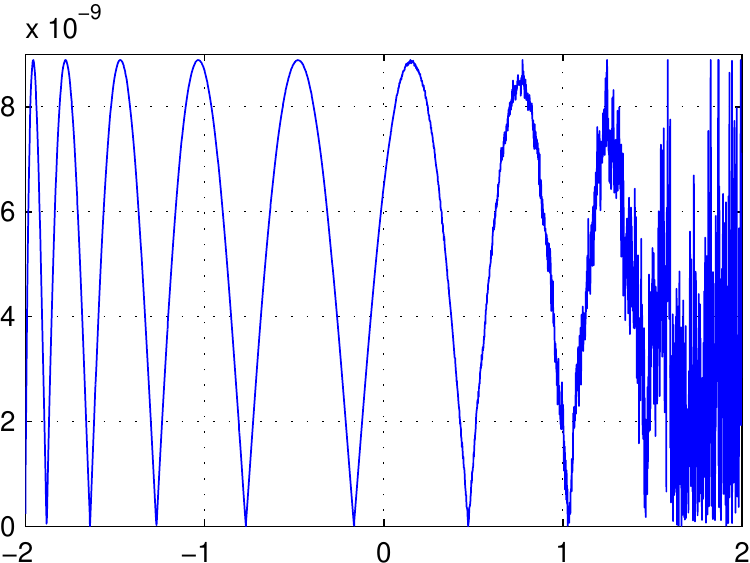}\quad
\includegraphics[
height=2.25in,
width=3in
]%
{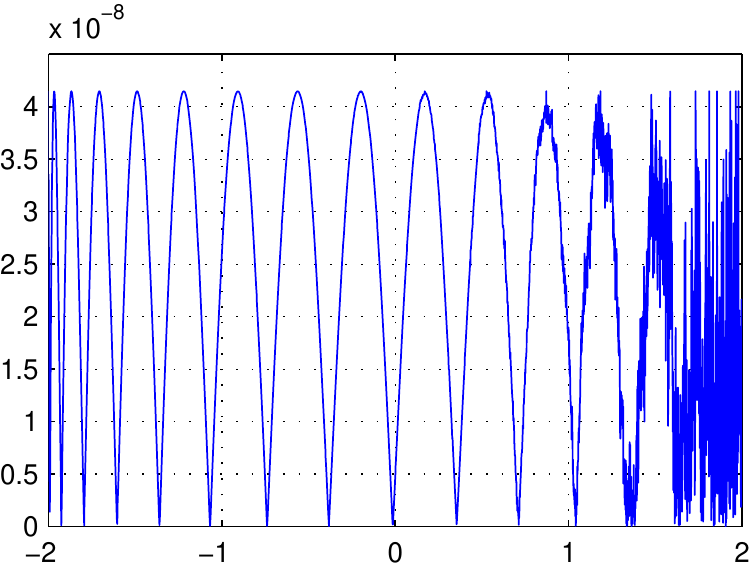}%
\caption{Graphs of $\left\vert g-P_{11}\right\vert $ (on the left) and $\left\vert h-Q_{18}\right\vert $
(on the right) from Example \ref{Ex1}
with the coefficients $\alpha_{k}$ and $\beta_{k}$ obtained by means of the Remez algorithm.}%
\label{Fig2}%
\end{center}
\end{figure}

The distribution of the absolute error of approximation of the solution of the
Cauchy problem (\ref{Cauchy1Ex1}), (\ref{Cauchy2Ex1}) is presented on Figure
\ref{Fig3}. Typically for an approximation based on a Taylor expansion
(generalized or not) the absolute error increases with the distance from the center.%

Obviously, neither always the expansion coefficients of the generalized Taylor
series of the initial data are available in a closed form nor always a
continuous function is representable in the form of such a series. In Section
\ref{SectApprox} several other possibilities for approximating functions by
$f$-polynomials were discussed. In the present example alternatively to the
generalized Taylor expansion we also apply the Remez algorithm (with the
single exchange method). The developed computer program in Matlab establishes
the corresponding value of $n$ for approximating a function by an
$f$-polynomial after which (for $n+1$, $n+2$, etc.) the approximation cannot
be significantly improved limited by the machine precision. Thus, for the
considered example the function $g$ was approximated by $P_{11}$ meanwhile $h$
was approximated by $Q_{18}$. Figure \ref{Fig2} depicts the
distribution of the corresponding absolute error of approximation. The maximum
value of the absolute error for $g$ was of order $10^{-9}$ and for $h$ --
$10^{-8}$.%

The distribution of the absolute error of approximation of the solution of the
Cauchy problem (\ref{Cauchy1Ex1}), (\ref{Cauchy2Ex1}) is presented on Figure
\ref{Fig3}. The maximum absolute error of the approximate solution is of the
order $10^{-8}$ and to the difference from the solution computed previously
with the use of the generalized Taylor coefficients here the distribution of
the error over the domain is more uniform.%

\begin{figure}[ptb]
\noindent
\begin{center}
\includegraphics[
height=2.5in,
width=3in
]%
{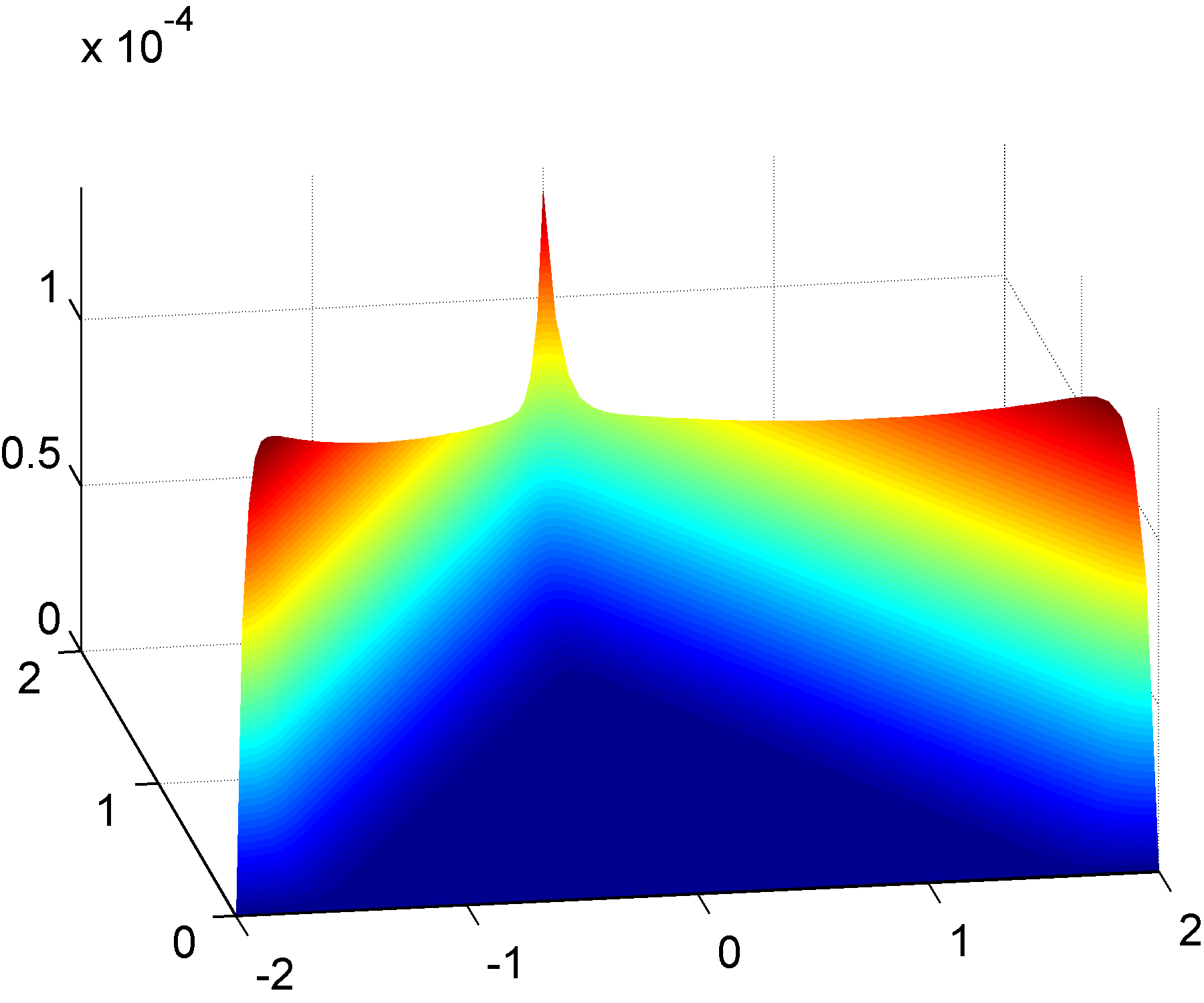}\quad
\includegraphics[
height=2.5in,
width=3in
]{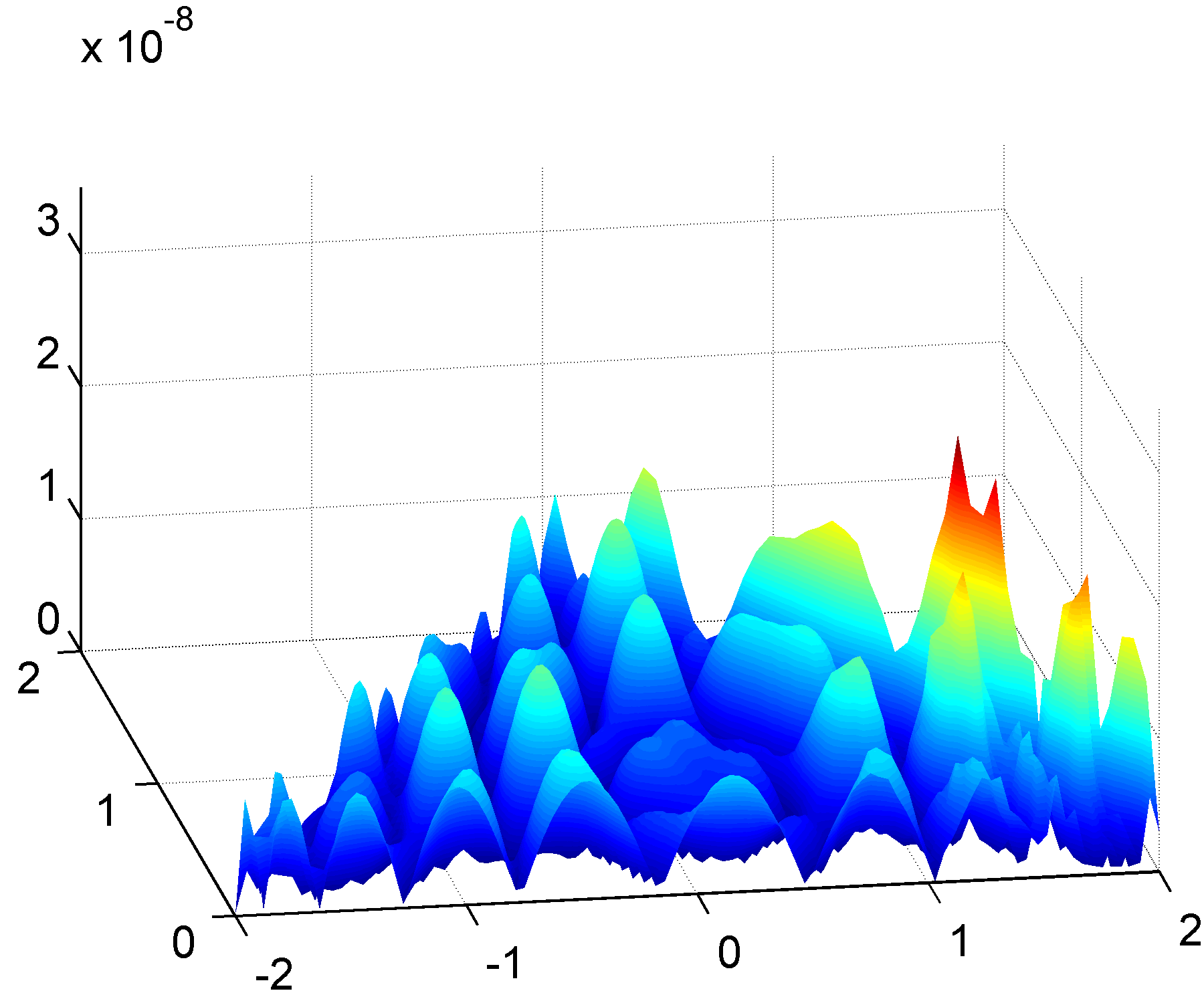}
\caption{The distribution of the absolute error of the approximate solutions of the Cauchy problem from Example \ref{Ex1} computed according to Proposition
\ref{Prop Approx Sol Estimate} with the $f$-polynomials of order $n=20$ and
generalized Taylor coefficients (on the left) and with the $f$-polynomials of
order $11$ (for $g$)  and $18$ (for $h$) with the coefficients obtained by the Remez
algorithm (on the right).}%
\label{Fig3}%
\end{center}
\end{figure}
\end{example}

\begin{example}
\label{Ex2} In this example we consider the same problem (\ref{Cauchy1Ex1}),
(\ref{Cauchy2Ex1}), again with $b=2$ and $\lambda_{1}=1$ but now with another
value of $c$, $c=5i$. Again we have the generalized Taylor coefficients in a
closed form (\ref{alphaKEx1}) and (\ref{betaKEx1}) but application of the
Remez algorithm encounters an obstacle. As the function $f(x)=e^{cx}$ is
complex valued (see Remark \ref{HaarFails}) the same is true for the functions
$\varphi_{k}$ meanwhile as was explained in Section \ref{SectApprox} the Remez
algorithm is directly applicable only to real valued functions.

Here in order to find the coefficients of the $f$-polynomials from Proposition
\ref{Prop Approx Sol Estimate} by a distinct from the generalized Taylor
formula method we solve the corresponding linear programming problem as
explained at the end of Section \ref{SectApprox}.

How well the approximation based on the generalized Taylor formula works is
shown in Figures \ref{Fig4} and \ref{Fig6} where the distribution of the absolute error of
approximation is depicted for $g$, $h$ and the solution $u$ respectively in
the case $n=50$.%

\begin{figure}[ptb]
\begin{center}
\noindent \includegraphics[
height=2.25in,
width=3in
]
{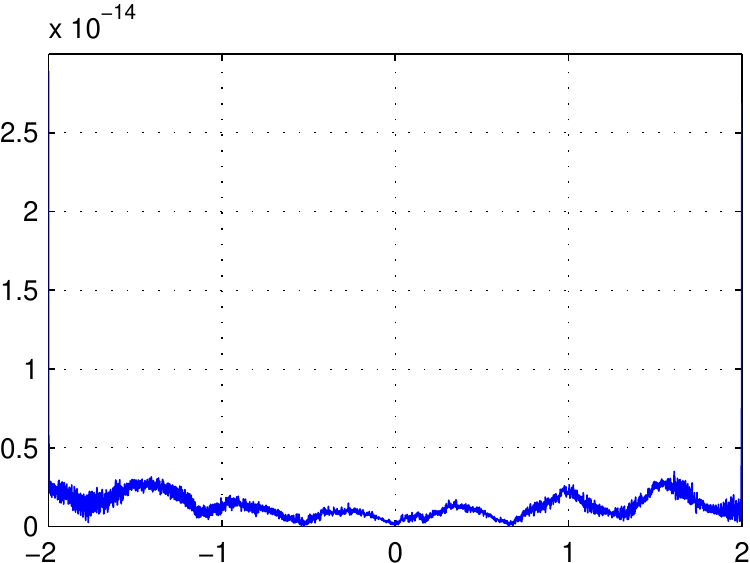}\quad
\includegraphics[
height=2.25in,
width=3in
]%
{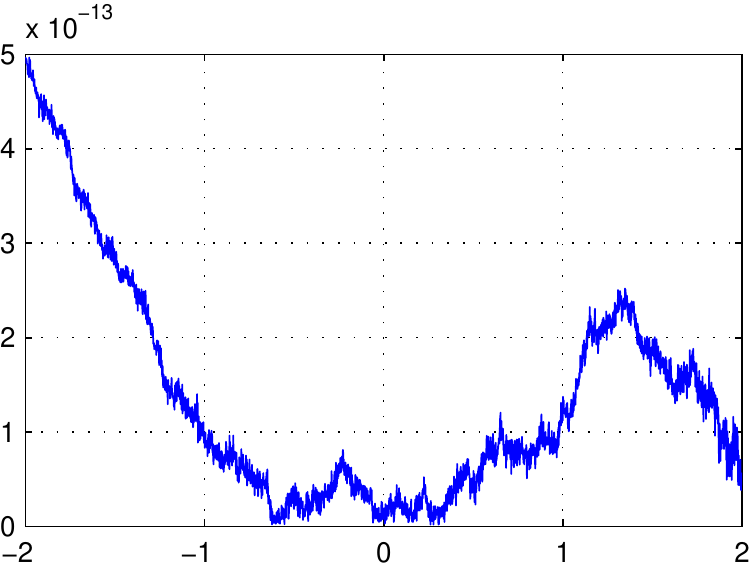}%
\caption{Graphs of $\left\vert g-P_{50}\right\vert $ (on the left) and $\left\vert h-Q_{49}\right\vert $ (on the right) from Example \ref{Ex2}
with the coefficients $\alpha_{k}$ and $\beta_{k}$ obtained by (\ref{alphaK}) which in this
case reduces to (\ref{alphaKEx1}) and (\ref{betaKEx1}).}%
\label{Fig4}%
\end{center}
\end{figure}

The corresponding results obtained by solving the linear programming problem
(again, due to the limitation of the machine precision we used $n=14$ for the approximation of $g$ and $n=25$ for the approximation of $h$) are presented in Figures \ref{Fig5} and \ref{Fig6}.

\begin{figure}[ptb]
\begin{center}
\noindent \includegraphics[
height=2.25in,
width=3in
]
{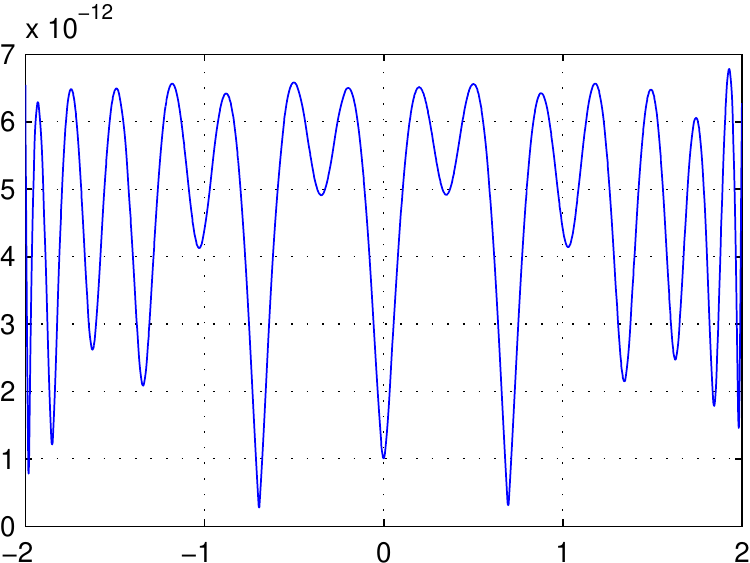}\quad
\includegraphics[
height=2.25in,
width=3in
]%
{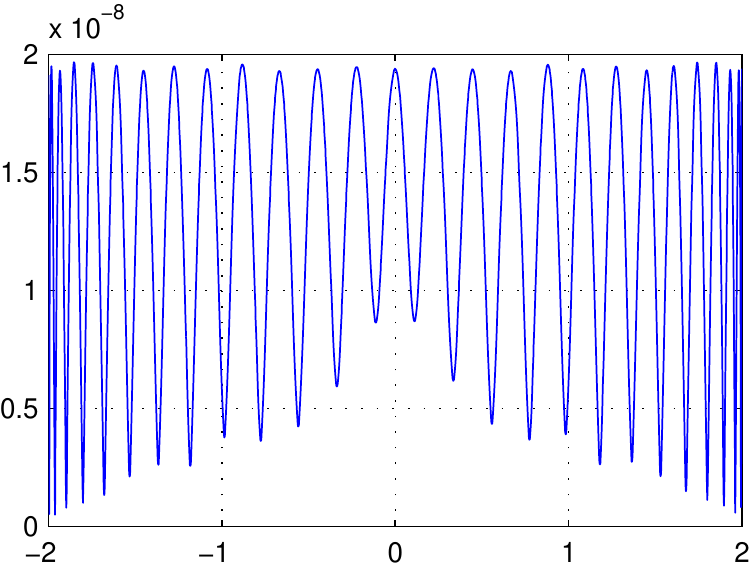}%
\caption{Graphs of $\left\vert g-P_{14}\right\vert $ (on the left) and $\left\vert h-Q_{25}\right\vert $ (on the right) from Example \ref{Ex2}
with the coefficients $\alpha_{k}$ and $\beta_{k}$ obtained by solving a linear programming
problem.}
\label{Fig5}
\end{center}
\end{figure}

\begin{figure}[ptb]
\begin{center}
\noindent\includegraphics[
height=2.5in,
width=3in
]{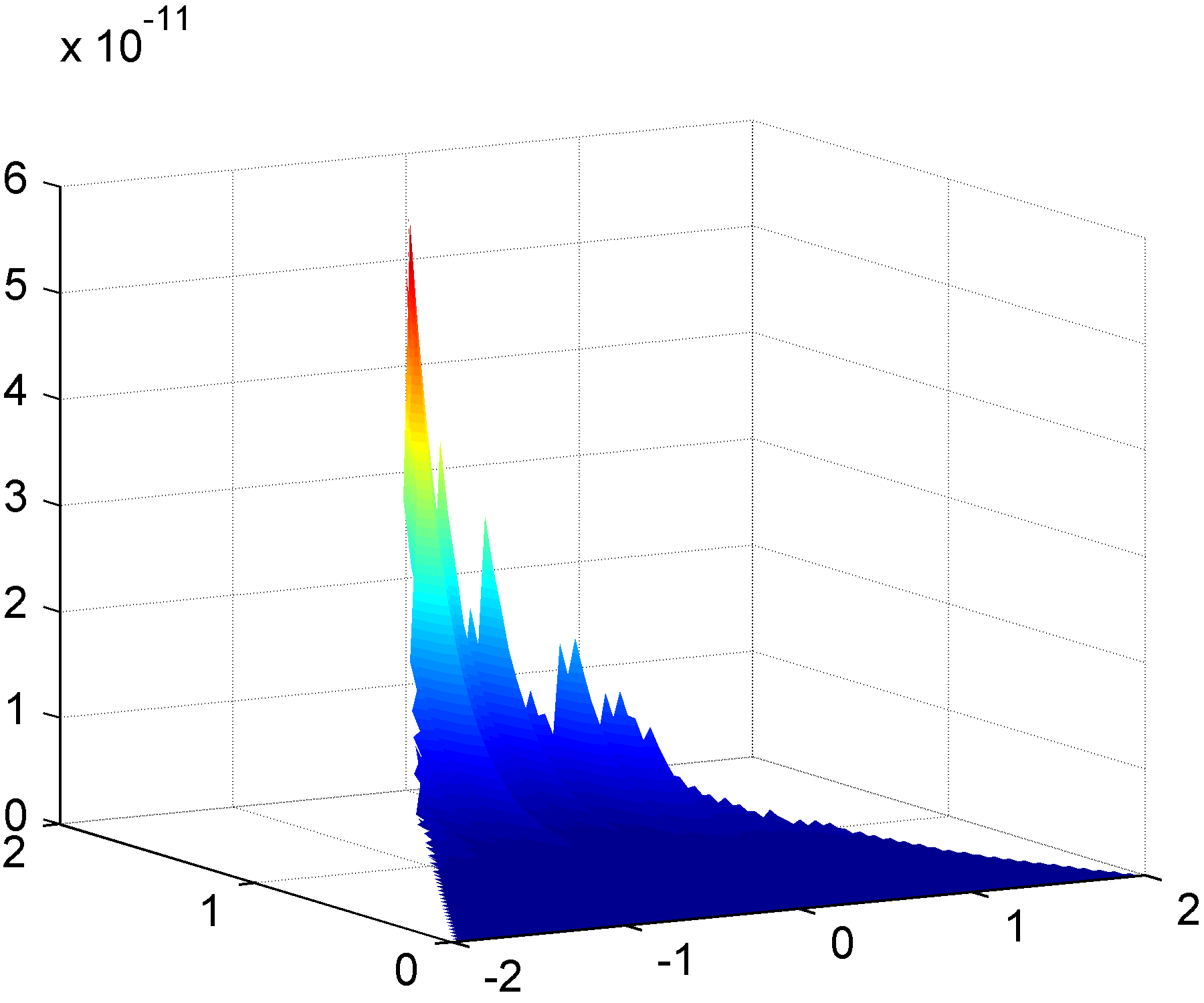}\quad
\includegraphics[
height=2.5in,
width=3in
]%
{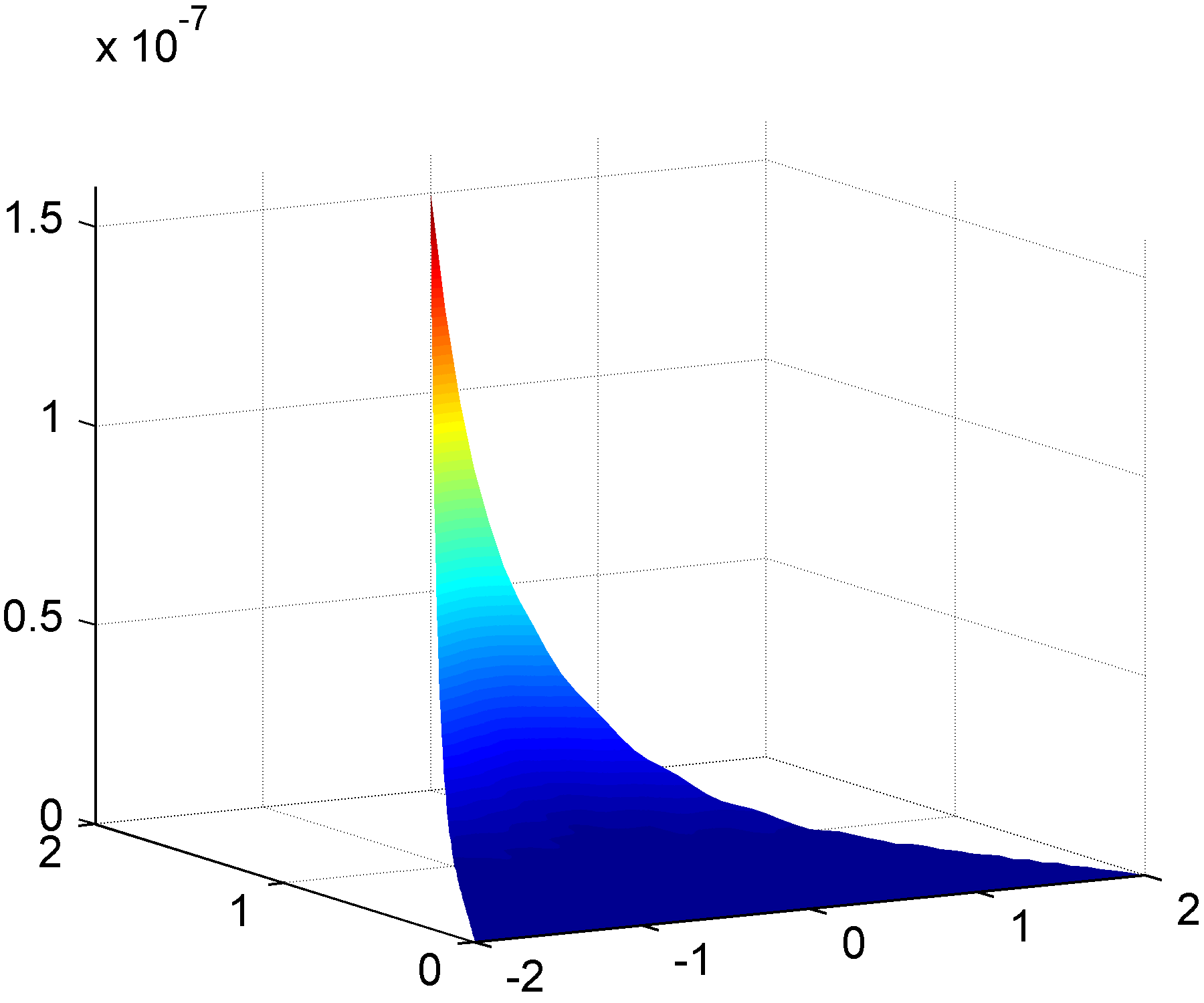}%
\caption{The distribution of the absolute error of the approximate solutions of the Cauchy problem from Example \ref{Ex2} computed according to Proposition
\ref{Prop Approx Sol Estimate} with the $f$-polynomials of order $n=50$ and
generalized Taylor coefficients (on the left) and with the $f$-polynomials of
order $14$ (for $g$)  and $25$ (for $h$) with the
coefficients obtained by solving a linear programming problem (on the right).}%
\label{Fig6}%
\end{center}
\end{figure}
\end{example}

\begin{example}
\label{Ex3} In our final example a variable coefficient equation is
considered. Namely, we solve the following Cauchy problem
\begin{align}
\square u-x^{2}u  &  =0,\quad-b\leq x\leq b,\quad t\geq0,\label{Cauchy1Ex3}\\
u(x,0)  &  =g(x)=e^{x^{2}/2}\biggl(1+\int_{0}^{x}e^{-s^{2}}ds\biggr),\label{Cauchy2Ex3}\\
u_{t}(x,0)  &  =h(x)=xe^{x^{2}/2}. \label{Cauchy3Ex3}%
\end{align}
Its exact solution is given by the expression%
\[
u(x,t)=g(x) \cosh t+h(x) \sinh\sqrt{3}t
\]
The corresponding ordinary second-order equation has the form
\begin{equation}
f^{\prime\prime}-x^{2}f=0. \label{Ex3EqF}%
\end{equation}
The functions $g$ and $h$ are solutions of the equation
\[
v^{\prime\prime}-x^{2}v=\lambda v\text{ }%
\]
with $\lambda=1$ and $\lambda=3$, respectively and hence again we have in our
disposal the possibility to write down the coefficients $\alpha_{k}$ and
$\beta_{k}$ from (\ref{g and h as series}) explicitly. We compute a particular
solution $f$ of (\ref{Ex3EqF}) numerically using the SPPS method as described
in \cite{KrPorter2010}, then compute the functions $\varphi_{k}$ and the
corresponding coefficients $\alpha_{k}$ and $\beta_{k}$ analogously to Example
\ref{Ex1}. The result for $n=20$ is presented by Figures \ref{Fig7} and \ref{Fig8} where
the error of approximation of $g$, $h$ and $u$ is depicted.%

\begin{figure}[ptb]
\noindent\begin{center}
\includegraphics[
height=2.25in,
width=3in
]%
{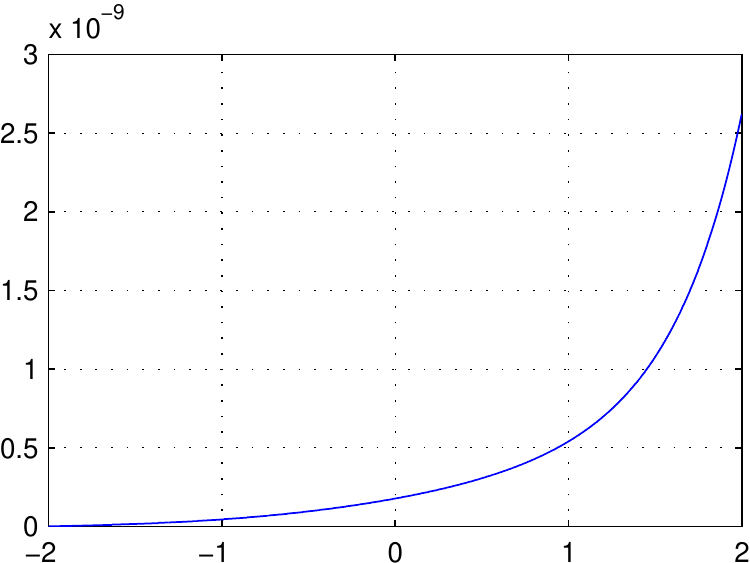}\quad
\includegraphics[
height=2.25in,
width=3in
]%
{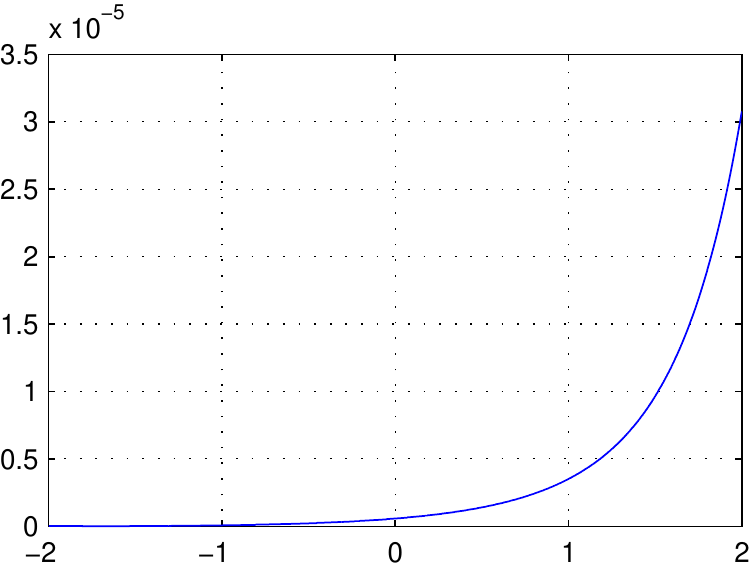}%
\caption{Graphs of $\left\vert g-P_{20}\right\vert $ (on the left) and $\left\vert h-Q_{19}\right\vert $ (on the right)
 from Example \ref{Ex3}
with the exact generalized Taylor coefficients (\ref{alphaK}). }
\label{Fig7}%
\end{center}
\end{figure}

Application of the Remez algorithm delivers the following results. The
functions $g$ and $h$ were approximated by $f$-polynomials of order $16$ and
$19$ respectively and the distribution of the absolute error of the
approximate solution of (\ref{Cauchy1Ex3})--(\ref{Cauchy3Ex3}) is depicted on
Figure \ref{Fig8}. As can be observed with this relatively small number of functions
$\varphi_{k}$ involved in the approximation a remarkable accuracy in the final
solution is achieved of the order $10^{-14}$.%

\begin{figure}[ptb]
\begin{center}
\noindent\includegraphics[
height=2.5in,
width=3in
]%
{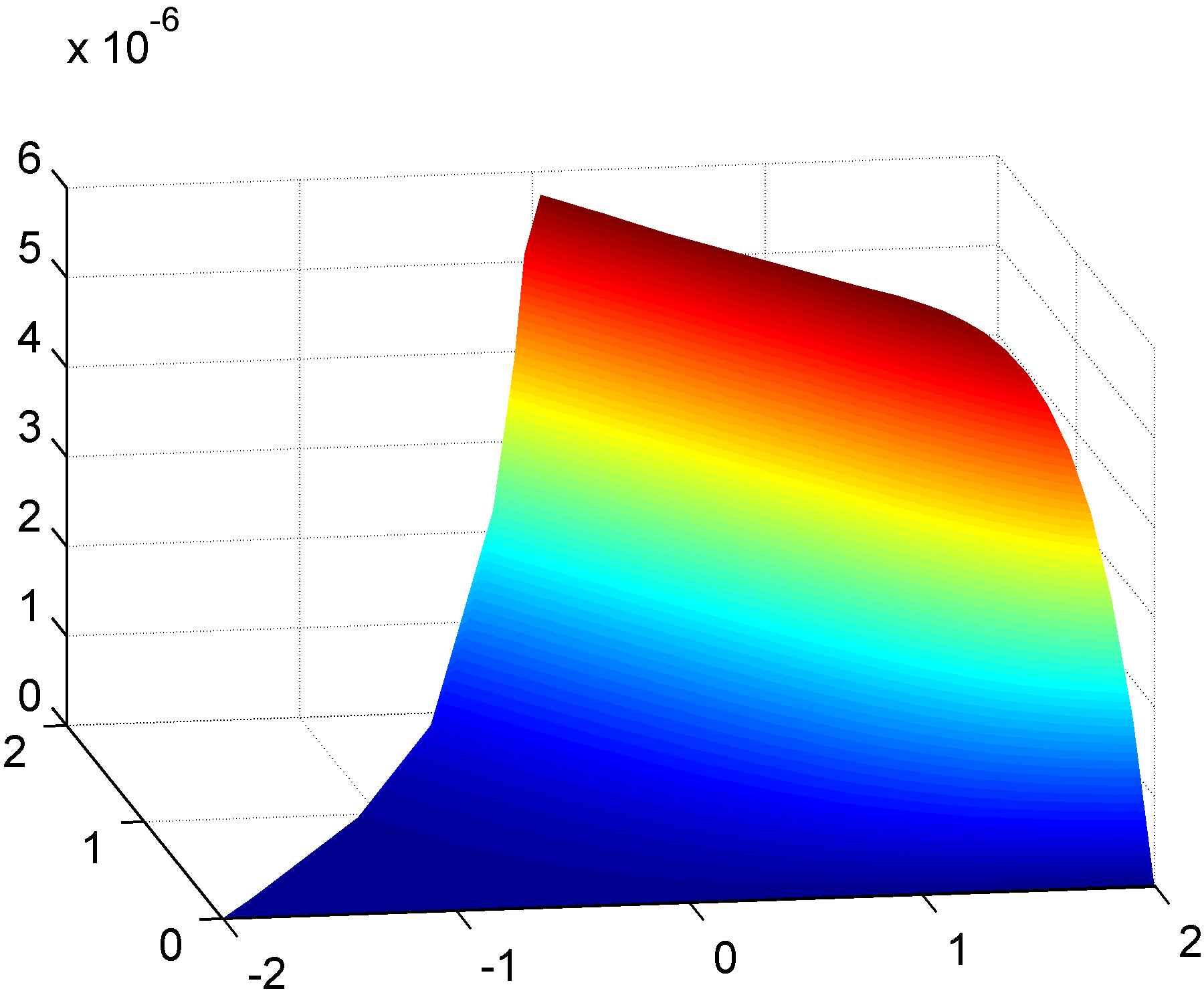}\quad
\includegraphics[
height=2.5in,
width=3in
]%
{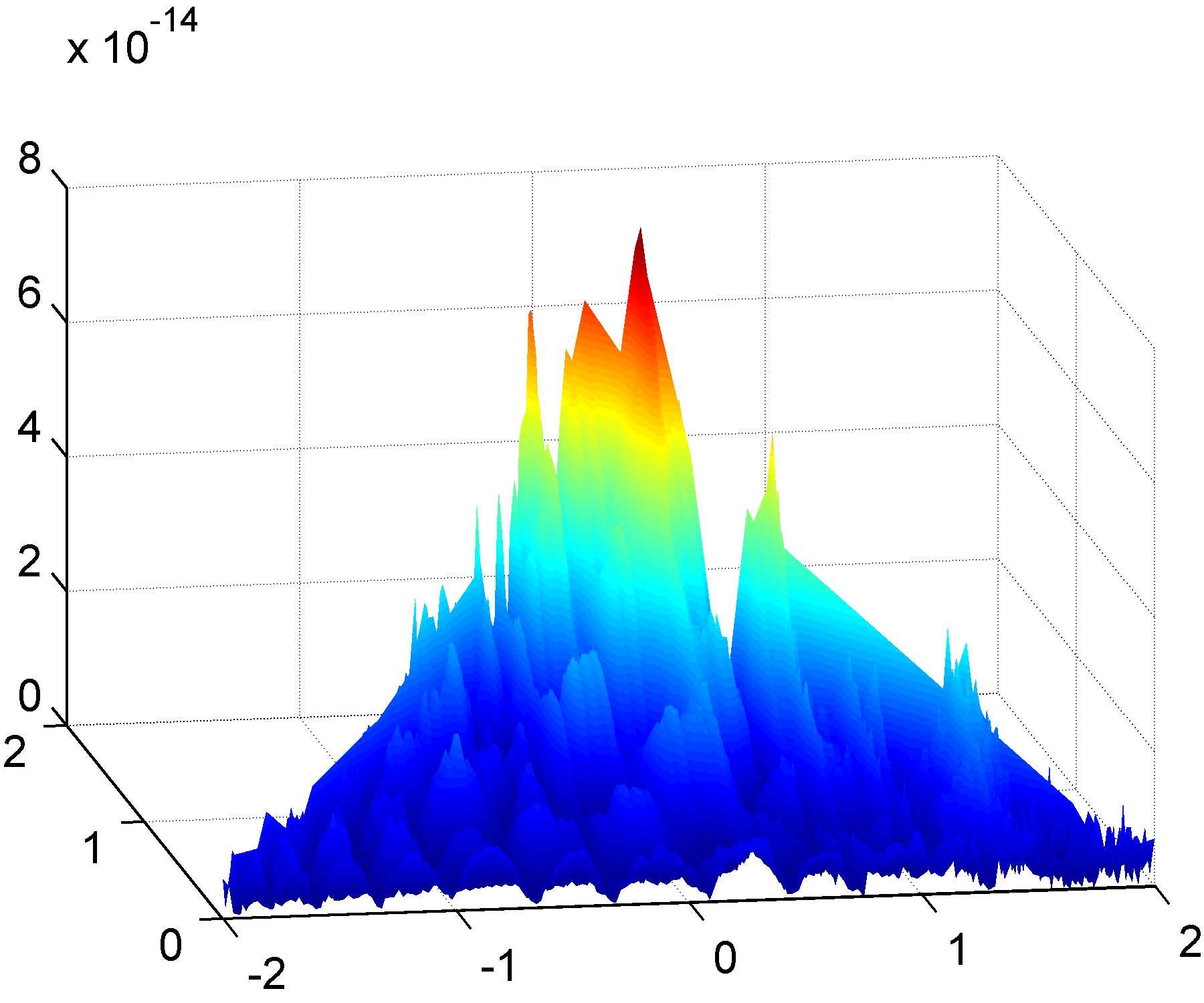}
\caption{The distribution of the absolute error of the approximate solutions
of the Cauchy problem from Example \ref{Ex3} computed according to Proposition
\ref{Prop Approx Sol Estimate} with the $f$-polynomials of order $n=20$ and
generalized Taylor coefficients (on the left) and with the $f$-polynomials of order $16$ (for
$g$) and $19$ (for $h$) with the coefficients obtained by the Remez
algorithm (on the right).}%
\label{Fig8}%
\end{center}
\end{figure}

\end{example}


\begin{thebibliography}{99}
{\small
\itemsep= 0pt

\bibitem {BDM1978}{\small I. Barrodale, L. M. Delves and J. C. Mason Linear
Chebyshev Approximation of Complex-Valued Functions. Math. Comput., 32 (1978),
853--863.}

\bibitem {Berskniga}{\small L. Bers Theory of pseudo-analytic functions (New
York University, 1952).}

\bibitem {CKT}{\small H. Campos, V. V. Kravchenko, S. M. Torba Transmutations,
L-bases and complete families of solutions of the stationary Schr\"{o}dinger
equation in the plane. J. Math. Anal. Appl. 389 (2012), no. 2, 1222--1238.}

\bibitem {Cheney}{\small E. W. Cheney, Introduction to Approximation Theory,
2nd ed., Chelsea, New York, 1986.}

\bibitem {DelsarteLions1956}{\small J. Delsarte and J. L. Lions,
Transmutations d'op\'{e}rateurs diff\'{e}rentiels dans le domaine complexe.
Comment. Math. Helv. 32 (1956), 113--128.}

\bibitem {DeVoreLorentz}{\small R. A. DeVore and G. G. Lorentz  Constructive
Approximation. Berlin: Springer-Verlag, 1993, x + 449 p.}

\bibitem {Dzyadyk}{\small V. K. Dzyadyk Introduction to the theory of uniform
approximation of functions by polynomials. Moscow: Nauka, 1977 (in Russian).}

\bibitem {Fage}{\small M. K. Fage and N. I. Nagnibida The problem of
equivalence of ordinary linear differential operators. Novosibirsk: Nauka,
1987 (in Russian).}

\bibitem {FM1992}{\small B. Fischer and J. Modersitzki An algorithm for
complex linear approximation based on semi-infinite programming. Numerical
Algorithms 5 (1993) 287--297.}

\bibitem {GR1981}{\small K. Glashoff and K. Roleff A new method for Chebyshev
approximation of complex-valued functions, Math. Comput., 36 (1981),
233--239.}

\bibitem {KelleyPeterson}{\small W. G. Kelley, A. C. Peterson  The Theory of
Differential Equations: Classical and Qualitative. Springer, 2010.}

\bibitem {Kovtunets1987}{\small V. V. Kovtunets Algorithm for computing the
best approximation polynomial of the complex-valued function. Issledovaniya po
teorii approximacii funkcij (Researchs on Function Approximation Theory),
Kiev, Inst. of Mathematics Publ., 1987, 35--42 (in Russian).}

\bibitem {Kovtunets1988}{\small V. V. Kovtunets Algorithm for computing the
best approximation polynomial of the complex-valued function on the compact
set, Nekotorye voprosy teorii priblizheniya funkcij i ih prilozheniya (Some
problems of approximation theory and its applications), Kiev, Inst. of
Mathematics Publ., 1988, 71--78 (in Russian).}

\bibitem {KrCV08}{\small V. V. Kravchenko A representation for solutions of
the Sturm-Liouville equation. Complex Variables and Elliptic Equations, 2008,
v. 53, 775-789.}

\bibitem {APFT}{\small V. V. Kravchenko Applied pseudoanalytic function
theory. Basel: Birkh\"{a}user, Series: Frontiers in Mathematics, 2009.}

\bibitem {KrCMA2011}{\small V. V. Kravchenko On the completeness of systems of
recursive integrals. Communications in Mathematical Analysis, Conf. 03, 2011,
172--176.}

\bibitem {KMoT}{\small V. V. Kravchenko, S. Morelos and S. Tremblay Complete
systems of recursive integrals and Taylor series for solutions of
Sturm-Liouville equations. Mathematical Methods in the Applied Sciences, v.
35, 2012, issue 6, 704--715.}

\bibitem {KrPorter2010}{\small V. V. Kravchenko and R. M. Porter Spectral
parameter power series for Sturm-Liouville problems. Mathematical Methods in
the Applied Sciences 2010, v. 33, 459-468.}

\bibitem {KRT}{\small V. V. Kravchenko, D. Rochon, S. Tremblay On the
Klein-Gordon equation and hyperbolic pseudoanalytic function theory. J. of
Phys. A, 2008, v. 41 issue 6, 065205.}

\bibitem {KT}{\small V. V. Kravchenko and S. M. Torba  Transmutations for
Darboux transformed operators with applications. J Phys A, v. 45, 2012, issue
7, \# 075201 (21 pp.)}

\bibitem {KT Obzor}{\small V. V. Kravchenko and S. M. Torba Transmutations and
spectral parameter power series in eigenvalue problems. To appear in Operator
Theory: Advances and Applications.}

\bibitem {Lavrentiev and Shabat}{\small M. A. Lavrentiev and B. V. Shabat
Problems of hydrodynamics and their mathematical models. Moscow: Nauka, 1973
(in Russian).}

\bibitem {LevitanInverse}{\small B. M. Levitan Inverse Sturm-Liouville
problems. VSP, Zeist, 1987.}

\bibitem {Lions57}{\small J. L. Lions Solutions \'{e}l\'{e}mentaires de
certains op\'{e}rateurs diff\'{e}rentiels \`{a} coefficients variables. Journ.
de Math. 36 (1957), Fasc 1, 57--64.}

\bibitem {Marchenko}{\small V. A. Marchenko Sturm-Liouville operators and
applications. Basel: Birkh\"{a}user, 1986.}

\bibitem {Matveev}{\small V. Matveev and M. Salle Darboux transformations and
solitons. N.Y. Springer, 1991.}

\bibitem {Meinardus}{\small G. Meinardus Approximation of Functions: Theory
and Numerical Methods, New York: Springer, 1967. Expanded English translation
of the German version: Approximation von Funktionen und ihre Numerische
Behandlung. Springer Tracts in Natural Philosophy, Volume 4, 1964.}

\bibitem {Motter and Rosa 1998}{\small A. E. Motter and M. A. Rosa  Hyperbolic
calculus. Advances in Applied Clifford Algebras, 1998, v. 8, No. 1, 109-128.}

\bibitem {Pinchover}{\small Y. Pinchover and J. Rubinstein An introduction to
partial differential equations. Cambridge University Press, 2005.}

\bibitem {Polyanin}{\small A. D. Polyanin Handbook of linear partial
differential equations for engineers and scientists. Boca Raton: Chapman \&
Hall/CRC, 2002.}

\bibitem {Remez}{\small E. Ja. Remez Fundamentals of numerical methods of
Tchebyshev approximation. Kiev: Naukova dumka, 1969 (in Russian).}

\bibitem {Rice}{\small J. R. Rice The Approximation of Functions, Vol. 1.
Linear theory. Reading, Massachusetts: Addison-Wesley, 1964.}

\bibitem {Rivlin}{\small T. J. Rivlin An introduction to the approximation of
functions. Blaisdell: Waltham, Mass., 1969.}

\bibitem {RS1961}{\small T. J. Rivlin and H. S. Shapiro A unified approach to
certain problems of approximation and minimization, J. Soc. Indust. Appl.
Math., 9 (1961), 670--699.}

\bibitem {SmirLeb}{\small V. I. Smirnov and N. A. Lebedev Functions of Complex
Variables: Constructive Theory, Moskow: Nauka, 1964 (in Russian). English
translation in V. I. Smirnov and N. A. Lebedev. Functions of Complex
Variables: Constructive Theory, MIT Press, Cambridge, MA, 1968.}

\bibitem {Tang1988}{\small P. T. P. Tang A fast algorithm for linear complex
Chebyshev approximation. Math. Comput., 52 (1988), 721--739.}

\bibitem {Timan}{\small A. F. Timan Theory of Approximation of Functions of a
Real Variable, Moskow, 1960 (in Russian). English translation in A. F. Timan,
Theory of Approximation of Functions of a Real Variable, New York: Macmillan,
1963.}

\bibitem {Trimeche}{\small K. Trimeche Transmutation operators and
mean-periodic functions associated with differential operators. London:
Harwood Academic Publishers, 1988.}

\bibitem {Vladimirov}{\small V. S. Vladimirov Equations of mathematical
physics. Moskva: Nauka, 1981 (in Russian), English translation in V.S.
Vladimirov. Equations of mathematical physics (2nd English ed.), Moscow: Mir
Publishers, 1983.}

\bibitem {Watson1990}{\small G. A. Watson Numerical methods for Chebyshev
approximation of complex-valued functions, in Algorithms for Approximation II,
J. C. Mason and M. G. Cox, eds., Chapman and Hall, London, 1989, 246--264.}
}
\end{thebibliography}
\end{document}